\documentclass{article}

\usepackage[latin1]{inputenc}
\usepackage{subfigure}
\usepackage{amsthm}
\usepackage{amsmath}
\usepackage{amsfonts}
\usepackage{wasysym}
\usepackage{enumerate}
\usepackage[margin=1in]{geometry}
\usepackage{bigints}
\usepackage[usenames]{color}
\usepackage{setspace}
\usepackage{tikz}

\newtheorem{lem}{Lemma}
\newtheorem{cor}[lem]{Corollary}
\newtheorem{thm}[lem]{Theorem}

\newtheorem{defn}[lem]{Definition}

\numberwithin{equation}{section}

\DeclareMathOperator{\arm}{arml}
\DeclareMathOperator{\leg}{legl}
\DeclareMathOperator{\sqr}{Sq}
\DeclareMathOperator{\inv}{inv}
\DeclareMathOperator{\NW}{NW}
\DeclareMathOperator{\wt}{wt}
\DeclareMathOperator{\hinv}{hinv}
\DeclareMathOperator{\vinv}{vinv}
\DeclareMathOperator{\ainv}{ainv}

\newcommand{\bfi}{\begin{figure} \begin{center}}
\newcommand{\efi}{\end{center} \end{figure}}
\newcommand{\capt}{\caption}
\newcommand{\hs}[1]{\hspace{#1}}

\newcommand{\da}{\hs{-2pt}\downarrow}

\author{
Kenneth Barrese\\[-5pt]
\small Department of Mathematics, Michigan State University,\\[-5pt]
\small East Lansing, MI 48824-1027  {\tt baressek@math.msu.edu}\\[5pt]
Nicholas Loehr \thanks{This work was partially supported by a grant from the Simons Foundation
   (\#244398 to Nicholas Loehr).}\\[-5pt]
\small Department of Mathematics, Virginia Tech\\[-5pt]
\small Blacksburg, VA 24061-0123 {\tt nloehr@vt.edu}\\[-5pt]
\small and\\[-5pt]
\small Department of Mathematics, United States Naval Academy\\[-5pt]
\small Annapolis, MD 21402-5002 {\tt loehr@usna.edu}\\[5pt]
Jeffrey Remmel\\[-5pt]
\small Department of Mathematics, UCSD\\[-5pt]
\small La Jolla, CA, 92093-0112 {\tt jremmel@ucsd.edu}\\[5pt]
Bruce E. Sagan\\[-5pt]
\small Department of Mathematics, Michigan State University,\\[-5pt]
\small East Lansing, MI 48824-1027 {\tt sagan@math.msu.edu}
}

\title{Bijections on $m$-level Rook Placements}

\begin{document}
\maketitle

\begin{abstract}

Suppose the rows of a board are partitioned into sets of $m$ rows called levels.  An $m$-level rook placement is a subset of the board where no two squares are in the same column or the same level.  We construct explicit bijections to prove three theorems about such placements.   We start with two  bijections between  Ferrers boards having the same number of $m$-level rook placements.  The first generalizes  a map  by Foata and Sch\"{u}tzenberger and our proof applies to any Ferrers board. This bijection also preserves the $m$-inversion number statistic of an $m$-level rook placement, defined by Briggs and Remmel. The second generalizes work of Loehr and Remmel. This construction only works for a special class of Ferrers boards, but it yields a formula for calculating the rook numbers of these boards in terms of elementary symmetric functions. Finally we generalize another result of Loehr and Remmel giving a bijection between boards with the same hit numbers.  The second and third bijections involve the Involution Principle of Garsia and Milne.

\end{abstract}

\section{Introduction}
\label{sec:in}

Rook theory is the study of the numbers $r_{k}(B)$, which count the number of ways to place $k$ non-attacking rooks on a board $B$. It originated with Kaplansky and Riordan~\cite{kr:pra} who studied the connections between rook placements and elements of the symmetric group $S_n$. 
We will focus on a particular type of board:
a Ferrers board is a board where the columns are bottom justified and their heights form a weakly increasing sequence. Foata and Sch\"{u}tzenberger~\cite{fs:rpf} characterized the equivalence classes of Ferrers boards by a unique increasing representative. 
They did so by constructing explicit bijections between rook placements on boards in a class and rook placements on the unique representative board. The rook polynomial of a board is the generating function for the numbers $r_k(B)$ in the falling factorial basis for the ring of polynomials. The theorem of Foata and Sch\"{u}tzenberger was later proved as an elegant corollary to the Factorization Theorem of Goldman, Joichi, and White~\cite{gjw:rtI}, which gave a complete factorization of the rook polynomial of a Ferrers board over the integers. Loehr and Remmel~\cite{lr:rbr} constructed a bijection between rook placements on rook equivalent Ferrers boards using the Garsia-Milne Involution Principle~\cite{gm:mcb}, which also implied the Factorization Theorem.
Later in the paper, they presented a similar bijection for the sets counted by the hit numbers of rook equivalent Ferrers boards. Briggs and Remmel~\cite{br:mrn} generalized the notion of rook placements to $m$-level rook placements. These correspond to elements of $C_m \wr S_n$, the wreath product of the cyclic group of order $m$ with the symmetric group on $n$ elements, in the same way that traditional rook placements correspond to elements of $S_n$. Using these placements and the concept of flag descents developed by Adin, Brenti, and Roichman~\cite{abr:dnm}, Briggs and Remmel were able to generalize a formula of Frobenius to $C_m \wr S_n$.

The purpose of this paper is to generalize the bijection of Foata and Sch\"{u}tzenberger and those of Loehr and Remmel to $m$-level rook placements. The remainder of this section gives the background terminology necessary to begin this task. In Section~\ref{sec:FS} we generalize the bijection used by Foata and Sch\"{u}tzenberger. Although this bijection is the composition of many intermediary bijections, and is  therefore not direct, it does provide an explicit bijection between $m$-level rook placements on arbitrary $m$-level rook equivalent Ferrers boards.  
We will need this bijection again in Section~\ref{sec:LR hit number}. In Section~\ref{sec:q-analogues} we show that the bijection provided in Section~\ref{sec:FS} preserves the $m$-inversion number of an $m$-level rook placement, as defined by Briggs and Remmel.
In Section~\ref{sec:LR placement} we generalize a construction of Loehr and Remmel. In this case the bijection can only be specified for singleton boards, a subset of all Ferrers boards. However, the construction leads to an explicit calculation of the $m$-level rook numbers for such boards using elementary symmetric functions and Stirling numbers of the second kind. Furthermore, this bijection also preserves the $m$-inversion number of $m$-level rook placements. In Section~\ref{sec:LR hit number}, we generalize a second bijection of Loehr and Remmel, and in doing so prove that any two $m$-level rook equivalent Ferrers boards have the same hit numbers.  The last two bijections involve the Garsia-Milne Involution Principle~\cite{gm:mcb}. 
 Finally, in Section~\ref{sec:Open Problem} we present an open problem about counting the number of Ferrers boards in $m$-level rook equivalence classes.

A \emph{board} is any finite subset of $\mathbb{Z}^+ \times \mathbb{Z}^+$ where $\mathbb{Z}^+$ is the positive integers. Given an integer partition $0 \leq b_1 \leq b_2 \leq \dots \leq b_n$, the corresponding \emph{Ferrers board} is
$$ B = \{(i,j)\in \mathbb{Z}^+ \times \mathbb{Z}^+\  | \  \text{$1\le i\le n$ and $j \leq b_i$} \}. $$
Usually $B$ is denoted by $B=(b_1,b_2,\dots,b_n)$. Graphically, one represents a Ferrers board as an array of square cells, where the $i$th column contains $b_i$ cells. See the diagram on the left in Figure~\ref{place} for the board $(1,1,3,4)$. Throughout, we will use $(i,j)$ to denote the cell in the $i$th column and $j$th row of $B$. 
 Note that this is neither the English nor the French style of writing Ferrers diagrams, but is the standard convention in modern rook theory literature.  It is useful because we usually consider placing rooks on the board from left to right, and enumerating the number of such placements is facilitated by our convention.

\bfi
\begin{tikzpicture}[scale=.5]
\foreach \y in {0,1} 
   \draw (0,\y)--(4,\y);
\foreach \y in {2,3,} 
   \draw (2,\y)--(4,\y);
\foreach \x in {3,4}
   \draw (\x,0)--(\x,4);
\draw (0,0)--(0,1);
\draw (1,0)--(1,1);
\draw (2,0)--(2,3);
\draw (3,4)--(4,4);
\draw(-2,1.5) node{$B=(1,1,3,4)=$};
\end{tikzpicture}
\hs{50pt}
\begin{tikzpicture}[scale=.5]
\foreach \y in {0,1} 
   \draw (0,\y)--(4,\y);
\foreach \y in {2,3,} 
   \draw (2,\y)--(4,\y);
\foreach \x in {3,4}
   \draw (\x,0)--(\x,4);
\draw (0,0)--(0,1);
\draw (1,0)--(1,1);
\draw (2,0)--(2,3);
\draw (3,4)--(4,4);
\draw (1.5,.5) node{$R$};
\draw (2.5, 2.5) node{$R$};
\draw (3.5, 1.5) node{$R$};
\end{tikzpicture}
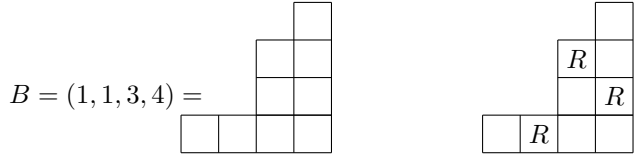
\capt{\label{place} A Ferrers board $B$ and a placement of three rooks on $B$.}
\efi

For any non-negative integer $k$, a \emph{placement of $k$ rooks} on $B$ is a subset of the cells of $B$ of cardinality $k$ which contains no more than one cell from any row or column of $B$. Graphically, this corresponds to placing rooks in the cells of $B$ so no two rooks attack  each other.   See the diagram on the right in Figure~\ref{place} for a placement of three rooks on $(1,1,3,4)$.

Henceforth we will assume that $m$ is a fixed positive integer. We define $\lceil j \rceil_m$ to be the least multiple of $m$ greater than or equal to $j$ and call it the \emph{$m$-ceiling} of $j$. Similarly, let $\lfloor j \rfloor_m$ be the greatest multiple of $m$ less than or equal to $j$, and call this the \emph{$m$-floor} of $j$. Given a positive integer $p$, let $L_p \subset \mathbb{Z}^+ \times \mathbb{Z}^+$ be defined by:
$$ L_p = \{(i,j)\in \mathbb{Z}^+ \times \mathbb{Z}^+\ |\ \lceil j \rceil_m = pm\}. $$
Then the \emph{$p$th level} of $B$ is $B \cap L_p$. Thus, the first level of $B$ consists of the first $m$ rows, the second level consists of the next $m$ rows, and so forth. Note that for $(i,j) \in B$, $\lceil j \rceil_m = pm$ if and only if $(i,j)$ is in the $p$th level of $B$.

For any non-negative integer $k$, an \emph{$m$-level rook placement of $k$ rooks} on $B$ is a subset of cardinality $k$ of the cells of $B$ which contains no more than one cell from any given level or column  of $B$.  See Figure~\ref{singletonfig} for three 2-level rook placements where thickened lines  demarcate where levels begin and end; the numbering of the boards can be ignored for now. An \emph{$m$-level rook} is a rook placed so that it is the only rook in its level and column. The \emph{$k$th $m$-level rook number of $B$} is
$$
r_{k,m}(B)=\text{the number of $m$-level rook placements of $k$ rooks on $B$.}
$$
Two boards are \emph{$m$-level rook equivalent} if their $m$-level rook numbers are equal for all $k$. Note that $m$-level rook placements are always rook placements.  Furthermore, when $m=1$ rook placements and $m$-level rook placements are equivalent.

The $i$th column of $B$ \emph{terminates in level $p$} if $p$ is the largest integer such that the $i$th column has non-empty intersection with $L_p$. A \emph{singleton board} is any Ferrers board such that, for each positive integer $p$, the set of all columns $b_i$  terminating in level $p$ contains at most one $i$ such that $b_i \not\equiv 0 \bmod{m}$. The Ferrers board on the left in Figure~\ref{singletonfig} is not a singleton board, as two different columns terminate in the second level without having $2$ cells in that level, while the Ferrers boards in the middle and on the right are singleton boards.

\bfi
\begin{tikzpicture}[scale=.7]
\draw[very thick] (0,0)--(4,0);
\foreach \y in {1} 
   \draw (0,\y)--(4,\y);
\draw[very thick] (1,2)--(4,2);
\foreach \y in {3} 
   \draw (1,\y)--(4,\y);
\foreach \x in {3,4}
   \draw (\x,0)--(\x,4);
\draw (0,0)--(0,1);
\draw (1,0)--(1,3);
\draw (2,0)--(2,3);
\draw[very thick] (3,4)--(4,4);
\draw(-1,1.5) node{$B=$};
\draw (2.5,1.5) node{$R$};
\draw (1.5, 2.5) node{$R$};
\draw (3.75,.75) node{$^1$};
\draw (3.75,1.75) node{$^2$};
\draw (2.75,.75) node{$^3$};
\draw (2.75,1.75) node{$^4$};
\draw (1.75,.75) node{$^5$};
\draw (1.75,1.75) node{$^6$};
\draw (.75,.75) node{$^7$};
\draw (3.75,2.75) node{$^1$};
\draw (3.75,3.75) node{$^2$};
\draw (2.75,2.75) node{$^3$};
\draw (1.75,2.75) node{$^4$};
\end{tikzpicture}
\hs{50pt}
\begin{tikzpicture}[scale=.7]
\draw[very thick] (0,0)--(4,0);
\foreach \y in {1} 
   \draw (0,\y)--(4,\y);
\draw[very thick] (1,2)--(4,2);
\foreach \y in {3} 
   \draw (2,\y)--(4,\y);
\foreach \x in {3,4}
   \draw (\x,0)--(\x,4);
\draw (0,0)--(0,1);
\draw (1,0)--(1,2);
\draw (2,0)--(2,4);
\draw[very thick] (2,4)--(4,4);
\draw(-1,1.5) node{$B_S=$};
\draw (1.5,1.5) node{$R$};
\draw (2.5, 3.5) node{$R$};
\draw (3.75,.75) node{$^1$};
\draw (3.75,1.75) node{$^2$};
\draw (2.75,.75) node{$^3$};
\draw (2.75,1.75) node{$^4$};
\draw (1.75,.75) node{$^5$};
\draw (1.75,1.75) node{$^6$};
\draw (.75,.75) node{$^7$};
\draw (3.75,2.75) node{$^1$};
\draw (3.75,3.75) node{$^2$};
\draw (2.75,2.75) node{$^3$};
\draw (2.75,3.75) node{$^4$};
\end{tikzpicture}
\hs{50pt}
\begin{tikzpicture}[scale=.7]
\foreach \y in {1,3}
    \draw (0,\y)--(2,\y);
\foreach \y in {0,2,4}
    \draw[very thick] (0,\y)--(2,\y);
\foreach \y in {5,7}
    \draw (1,\y)--(2,\y);
\draw[very thick] (1,6)--(2,6);
\draw (0,0)--(0,4);
\draw (1,0)--(1,7);
\draw (2,0)--(2,7);
\draw(-1,1.5) node{$l(B_S)=$};
\draw (.5,3.5) node{$R$};
\draw (1.5, 5.5) node{$R$};
\draw (1.75,.75) node{$^1$};
\draw (1.75,1.75) node{$^2$};
\draw (1.75,2.75) node{$^3$};
\draw (1.75,3.75) node{$^4$};
\draw (1.75,4.75) node{$^5$};
\draw (1.75,5.75) node{$^6$};
\draw (1.75,6.75) node{$^7$};
\draw (.75,.75) node{$^1$};
\draw (.75,1.75) node{$^2$};
\draw (.75,2.75) node{$^3$};
\draw (.75,3.75) node{$^4$};
\end{tikzpicture}
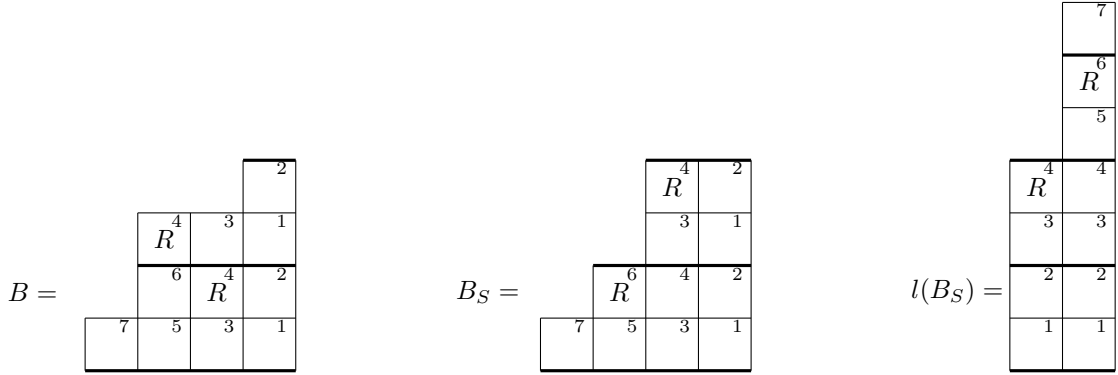
\capt{\label{singletonfig} On the left, a placement of two 2-level rooks on $B$. In the middle, the corresponding placement from Lemma~\ref{singleton} of two 2-level rooks on singleton board $B_S$. On the right, the placement on $l(B_S)$ from Lemma~\ref{l operator}.}
\efi

\section{Rook equivalence and bijections}
\label{sec:FS}

\subsection{Reduction to singleton boards}

In order to produce bijections between $m$-level rook placements on Ferrers boards, it is convenient to restrict our attention to singleton boards. In order to do this we prove the following two lemmas. First we show that for every Ferrers board there is a unique singleton board which has the same number of cells at each level. Then we prove that there is a bijection between the rook placements on a Ferrers board and those on the singleton board guaranteed in the first lemma. These lemmas together imply that every Ferrers board is $m$-level rook equivalent to a singleton board and that there is an explicit bijection between the corresponding rook placements.

\begin{lem}
\label{singletonexist}
Given a Ferrers board $B$, there exists a unique singleton board $B_S$ which has the same number of cells at each level as $B$.
\end{lem}

\begin{proof}
Let $B$ have $l_p$ cells in the $p$th level. In order for $B_S$ to be a singleton board with $l_p$ cells in the $p$th level, the cells of the $p$th level must be arranged uniquely as follows. If $l_p = cm+r$ with $0 \leq r < m$, then level $p$ of $B_S$ must have one column with $r$ cells followed on the right by $c$ columns with a full $m$ cells in the level. This is because a singleton board may have at most one column which intersects a given level non-trivially in fewer than $m$ cells. Thus $B_S$ must be unique if it exists.

In order to show that $B_S$ exists, we shall construct it. Arrange each level as specified above and line up the furthest right column in each level to create the furthest right column of $B_S$. This yields a Ferrers board because every column which has any cells in the $p$th level of $B$ must have a full $m$ cells in the $(p-1)$st level of $B$. Thus the total number of columns in the $p$th level of $B_S$ will be less than or equal to the number of columns in the $(p-1)$st level of $B_S$ containing $m$ cells at that level. Hence a singleton board $B_S$ exists and is unique.
\end{proof}

Ignoring the rook placement, Figure~\ref{singletonfig} shows a board $B$ and its corresponding board $B_S$. Since we know that an arbitrary Ferrers board $B$ has the same number of cells at each level as a unique singleton board $B_S$, we wish to provide an explicit bijection between rook placements on the two boards. In order to do so we require the following numbering on a Ferrers board.

\begin{defn}
\label{level numbering}
A \emph{level numbering} of board $B$ assigns a number to each cell of $B$ in the following way. Proceeding level by level in $B$, number the cells in the level by numbering each column from bottom to top, starting with the rightmost column and working left. In each level begin the numbering with 1. 
\end{defn}

Figure~\ref{singletonfig} presents two examples of this numbering, on the left and middle boards, and also illustrates the bijection of the next lemma. 

\begin{lem}
\label{singleton}
Given a Ferrers board $B$, there is an explicit bijection between $m$-level rook placements of $k$ rooks on $B$ and $m$-level rook placements of $k$ rooks on $B_S$, where $B_S$ is as constructed in Lemma~\ref{singletonexist}.
\end{lem}

\begin{proof}
Give both $B$ and $B_S$ a level numbering as shown in Figure~\ref{singletonfig}. Since both boards have the same number of cells in each level, corresponding levels will each be numbered with the same set of numbers. Given any $m$-level rook placement on $B$, place rooks on $B_S$ initially so that each rook occupies the same numbered cell in the same level as it does in $B$. This may not provide an $m$-level rook placement on $B_S$ since two rooks could end up in the same column, so we will modify it as follows. 

Notice that if a rook in column $i$ and level $p$ of $B$ is not in the same cell in $B_S$, then column $i$ must be to the left of a column of $B$ that intersects $L_p$ in less than $m$ cells. Furthermore, if the rook ends up in column $i'$ in $B_S$, then all  columns in the interval $[i,i']$ have a full $m$ cells in levels below $p$ in $B$.  Thus, if any of the rooks that move create a column with two or more rooks, there will be exactly two rooks in the column and the upper rook will have moved while the lower rook remained stationary.  To rectify the situation, whenever a rook is moved from column $i$ in $B$ to column $i'$ in $B_S$, move all other rooks in columns in the interval $(i,i']$ one column to the left, preserving their row. 
This is possible since, in both $B$ and $B_S$, these columns must contain $m$ cells in all levels lower than the upper rook in order for the upper rook to have been in that column in $B$. Rearranging the rooks at each level in this fashion provides a function from $m$-level rook placements on $B$ to $m$-level rook placements on $B_S$. Figure~\ref{singletonfig} illustrates this map on a rook placement, including moving a lower rook one column to the left.

To see that this is a bijection, use the level numbering to produce a set of rooks on $B$ from those on $B_S$.  All the rooks will return to their initial positions once the appropriate right shift is applied.  Similarly one can show that applying the map first to $B_S$ and then to $B$ is the identity.  Thus we have a bijection between $m$-level rook placements on $B$ and $m$-level rook placements on $B_S$.
\end{proof}

Lemma~\ref{singletonexist} and Lemma~\ref{singleton} guarantee that every Ferrers board is $m$-level rook equivalent to a singleton board. Additionally, there is an explicit bijection between $m$-level rook placements on the two boards. This permits us to restrict our attention to singleton boards henceforth.

\subsection{The $l$-operator}
\label{subsec:l op}

Transposition of boards plays a central role in the Foata-Sch\"utzenberger construction of bijections between rook-equivalent Ferrers boards when $m=1$.  We will need a generalization of this operation for arbitrary $m$ and this is given in the next definition.

\begin{defn}
Given a Ferrers board $B$, the \emph{$l$-operator} applied to $B$ is defined as follows. If $t$ is the largest index of a non-empty level of $B$ and the number of cells in  the $p$th level of $B$ is $l_p$, then
$$l(B) = (l_t, l_{t-1}, \dots , l_1).$$
\end{defn}

Figure~\ref{singletonfig} contains an example board $B_S$ as well as $l(B_S)$.  The fact that $l(B)$ is a Ferrers board comes from the proof of Lemma~\ref{singletonexist}.
In particular, if $B$ is a Ferrers board then its $p$th level must fit above its $(p-1)$st level which implies
$$
\lfloor l_p \rfloor_m\le  \lfloor l_{p-1} \rfloor_m,
$$
with strict inequality if $l_p \not\equiv 0 \bmod{m}$.
It follows  that $l(B)$ is a weakly increasing sequence and so
$l(B)$ is a Ferrers board and, because of the strict inequality for non-multiples of $m$, a singleton board.

To see that the $l$-operator is a generalization of transposition, note that if $m=1$ then the levels of $B$ are individual rows and these become the columns of $l(B)$. Furthermore, when restricted to the set of singleton boards the $l$-operator is an involution. This is shown in Proposition 7.4 of~\cite{blrs:mrp}.  Thus, the $l$-operator is a surjection from the set of Ferrers boards onto the set of singleton boards with $B=l(l(B))$ when $B$ is singleton. We now provide a bijection between $m$-level rook placements on $B$ and $m$-level rook placements on $l(B)$ to generalize the well-known bijection for transposition.

\bfi
\begin{tikzpicture}[scale=.5]
\fill[lightgray] (0,0) rectangle (3,1);
\fill[lightgray] (2,1) rectangle (3,2);
\draw[very thick] (0,0)--(4,0);
\foreach \y in {1} 
   \draw (0,\y)--(4,\y);
\draw[very thick] (2,2)--(4,2);
\foreach \y in {3} 
   \draw (2,\y)--(4,\y);
\foreach \x in {3,4}
   \draw (\x,0)--(\x,4);
\draw (0,0)--(0,1);
\draw (1,0)--(1,1);
\draw (2,0)--(2,3);
\draw[very thick] (3,4)--(4,4);
\draw[very thick, dashed] (3.5,2) -- (3.5,4);
\end{tikzpicture}
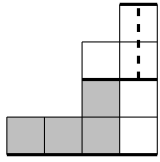
\capt{\label{armleg}  The dashed line goes through the cells counted by the $2$-arm length of the fourth column and first level, and the shaded cells are counted by the corresponding $2$-leg length.}
\efi

\begin{lem}
\label{l operator}
Given a singleton board $B$ and a non-negative integer $k$, there is an explicit bijection between $m$-level rook placements of $k$ rooks on $B$ and $m$-level rook placements of $k$ rooks on $l(B)$.
\end{lem}

\begin{proof}
Give $B$ a level numbering, then number the columns of $l(B)$ from bottom to top beginning with the number $1$ in each column. Note that in this case the numbering of a level of $B$ will consist of the same set of numbers as the numbering of the corresponding column of $l(B)$.
Assume that $B$ has $t$ non-empty levels.  For a given $m$-level rook placement of $k$ rooks on $B$, place rooks on $l(B)$ in the following way. If a rook was in the cell numbered $n$ of level $p$ in $B$, then place a rook in the cell numbered $n$ in column $t-p+1$ in $l(B)$.  See Figure~\ref{singletonfig} for an example of this map for a 2-level placement.

We must show that this gives a valid $m$-level rook placement on $l(B)$. If two rooks end up in the same column of $l(B)$ they must have originated in the same level of $B$, contradicting having an $m$-level rook placement on $B$. Similarly, if two rooks end up in the same level of $l(B)$, then they must have originated in the same column of $B$, since $B$ is a singleton board.

The inverse of this map acts as follows.  If a rook is in the cell numbered $a$ 
of column $t-p+1$ in $l(B)$ then it is placed in the cell numbered $a$ in level $p$ of $B$.  The proof that this gives a rook placement is similar to the one in the previous paragraph and so is omitted.
\end{proof}

Note that Lemma~\ref{singleton} and Lemma~\ref{l operator} combine to provide an explicit bijection between $m$-level rook placements on any Ferrers board $B$ and on its $m$-transpose, $l(B)=l(B_S)$.

\subsection{The local $l$-operator}

For any set $S$, let $\# S $ be the cardinality of  $S$. Given a column $i$ and a level $p$ define the \emph{$m$-arm length} of column $i$, level $p$ by
$$\arm_m(i,p) = \# \{ (i,j') \in B \ |\ (i,j') \text{ is strictly above level } p\}.$$
In Figure~\ref{armleg} the cells counted by the $2$-arm length of column $4$, level $1$ have a dashed line through them. (Reflecting our boards to put them in English notation will result in the arm being the usual set of squares when $m=1$.) We let $\arm_m(i,p)= \infty$ if the number of columns in $B$ is less than $i$, for reasons detailed in Lemma~\ref{perm:lem}.

Similarly, define the \emph{$m$-leg length} of column $i$, level $p$ to be
$$ \leg_m(i,p) = \# \{ (i',j') \in B \ |\  \text{$(i',j')$  is in level $p$ and $i'<i$}\}.$$
The cells counted by the $2$-leg length of column $4$, level 1 are shaded in Figure~\ref{armleg}. As before, this is equivalent to the usual notion of leg length in the $m=1$ case. We also let $\leg_m(i,0)= \infty$ by convention.

Since the $l$ operation generalizes the transposition of a Ferrers board, one would expect that some sort of local $l$ operation would be the appropriate generalization of the local transposition introduced by Foata and Sch\"{u}tzenberger. This is indeed the case, and we define the local $l$ operation as follows.

Given a Ferrers board $B$ with non-empty intersection of the $i$th column and $p$th level, let $B_{i,p}$ denote the subboard of $B$ consisting of all cells in or above the $p$th level and in or to the left of the $i$th column: see Figures~\ref{not permissible} and~\ref{permissible} for examples.   Note that if $B$ is a singleton board, then $B_{i,p}$ is also, because the set of rows in  level $p'$  of $B_{i,p}$ will be the same as the set of rows in level $p+p'-1$ of $B$.  If $B$ is a Ferrers board then the \emph{local $l$ operation at $(i,p)$} is the result of applying the $l$ operator to the subboard $B_{i,p}$ and leaving the rest of $B$ fixed. We will denote the resulting board by $l_{i,p}(B)$.

As defined above $l_{i,p}(B)$ may not be a Ferrers board, let alone a singleton board. We now develop a pair of conditions to determine if $l_{i,p}(B)$ will be a singleton board.

\bfi
\begin{tikzpicture}[scale=.5]
\draw (-1,2.5) node{$B=$};
\fill[lightgray] (1,2) rectangle (3,4);
\draw[very thick] (0,0)--(4,0);
\draw (0,1)--(4,1);
\foreach \y in {2,4}
    \draw[very thick] (1,\y)--(4,\y);
\draw (1,3)--(4,3);
\draw (3,5)--(4,5);
\draw (0,0)--(0,1);
\foreach \x in {1,2}
    \draw (\x,0)--(\x,4);
\foreach \x in {3,4}
    \draw (\x,0)--(\x,5);
\end{tikzpicture}
\hs{50pt}
\begin{tikzpicture}[scale=.5]
\fill[lightgray] (2,2) rectangle (3,6);
\draw (-1,2.5) node{$l_{3,2}(B)=$};
\draw[very thick] (0,0)--(4,0);
\draw (0,1)--(4,1);
\draw[very thick] (1,2)--(4,2);
\foreach \y in {3,5}
    \draw (2,\y)--(4,\y);
\draw[very thick] (2,4)--(4,4);
\draw[very thick] (2,6)--(3,6);
\draw (0,0)--(0,1);
\draw (1,0)--(1,2);
\foreach \x in {2,3}
    \draw (\x,0)--(\x,6);
\draw (4,0)--(4,5);
\end{tikzpicture}
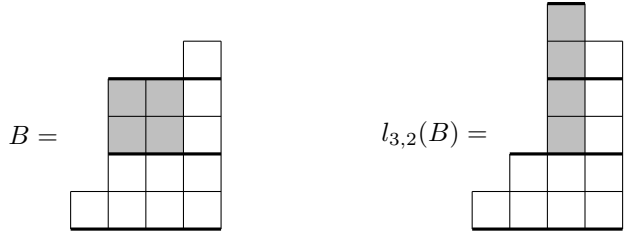
\capt{\label{not permissible} On the left, $B_{3,2}$ is shaded within $B=(1,4,4,5)$.  Notice that $l_{3,2}$ is not permissible for $B$ since $\lfloor \arm_m(4,2) \rfloor_m<\leg_m(3,2)$, which means $l_{3,2}(B)$ will not be a singleton board. On the right the shaded cells in $l(B_{3,2})$ illustrate this; in this case $l_{3,2}(B)$ is not even a Ferrers board.}
\efi

\begin{defn}
\label{perm:def}
The operation $l_{i,p}$ is \emph{permissible} for a singleton board $B$ if
$$
\arm_m(i,p) \leq \lfloor \leg_m(i,p-1) \rfloor_m
\quad \text{and}\quad
\leg_m(i,p) \leq \lfloor \arm_m(i+1,p) \rfloor_m.
$$
\end{defn}

See Figure~\ref{not permissible} for an example of a local $l$-operation not permissible for the given board, and Figure~\ref{permissible} for a local $l$-operation which is permissible.

\begin{lem}
\label{perm:lem}
Let a singleton Ferrers board $B$ have a non-empty intersection of the $i$th column and $p$th level.  
Then $l_{i,p}$ is  permissible for $B$ if and only if $l_{i,p}(B)$ is a singleton Ferrers board.
\end{lem}

\begin{proof}

If column $i$, level $p$ in $B$ contains fewer than $m$ cells,
then $l_{i,p}(B)=B$ since $B$ is singleton, and there is nothing to prove.
Henceforth, assume that column $i$, level $p$ in $B$ contains $m$ cells.
We know that $B$, $B_{i,p}$, and $l(B_{i,p})$ are all singleton Ferrers
boards. It follows that $l_{i,p}(B)$ will be a singleton Ferrers board
if and only if these three conditions hold for the board $l_{i,p}(B)$.
\begin{enumerate}
\item[(a)] The lowest row of level $p$ is weakly shorter than the highest row of level $p-1$;
\item[(b)] column $i$ is weakly shorter than column $i+1$; and
\item[(c)] if columns $i$ and $i+1$ terminate at the same level, then the height of column $i+1$ is a multiple of $m$. 
\end{enumerate}
Condition (c) is needed to ensure $l_{i,p}(B)$ will be singleton.

To determine when these conditions hold, first note that applying
$l_{i,p}$ to $B$ exchanges $\arm_m(i,p)$ and $\leg_m(i,p)$.
Because $B$ is singleton, the top row of level $p-1$ in $B$
(and in $l_{i,p}(B)$) extends left of column $i$ by $\lfloor \leg_m(i,p-1)\rfloor_m/m$ cells. On the other hand,
the new bottom row of level $p$ in $l_{i,p}(B)$ extends left of column $i$
by $\lceil \arm_m(i,p)\rceil_m/m$ cells. Thus, condition (a) holds if and only if
$$ \lceil\arm_m(i,p)\rceil_m \leq \lfloor\leg_m(i,p-1)\rfloor_m. $$
Since both sides are multiples of $m$, this inequality is equivalent to
$\arm_m(i,p)\leq \lfloor\leg_m(i,p-1)\rfloor_m$, which is the
first condition in the definition of permissibility.

Now consider the heights of columns $i$ and $i+1$ in $l_{i,p}(B)$.
Both column $i$ and column $i+1$ have a full $m$ cells in level $p$.
So, in both $B$ and $l_{i,p}(B)$, column $i+1$ extends above level $p$ by $\arm_m(i+1,p)$ cells.
On the other hand, the new column $i$ in $l_{i,p}(B)$ extends above level $p$ by
$\leg_m(i,p)$ cells. So condition (b) will hold if and only if
$$ \leg_m(i,p)\leq\arm_m(i+1,p). $$
To deal with condition (c), consider two cases.
First suppose that $\arm_m(i+1,p)$ is a multiple of $m$. Then condition (c) must hold, and here condition (b) will hold if and only if
$\leg_m(i,p)\leq\lfloor\arm_m(i+1,p)\rfloor_m$.  Now suppose that $\arm_m(i+1,p)$ is not a multiple of $m$.
Given that condition (b) holds, the new board $l_{i,p}(B)$ will be singleton
if and only if the strengthened inequality
$\leg_m(i,p)\leq\lfloor\arm_m(i+1,p)\rfloor_m$ is true.
Thus, this last inequality is equivalent to the truth of (b) and (c) in all cases.
\end{proof}

\subsection{The Local $l$-operation on an $m$-level rook placement}

Since there is a bijection between rook placements on $B$ and $l(B)$ when $B$ is singleton, it stands to reason that it would generalize to a bijection between rook placements on $B$ and $l_{i,p}(B)$. The following lemma makes this precise.

\begin{lem}
\label{local l}
For a singleton board $B$, suppose  $l_{i,p}$ is permissible for $B$. Then there is an explicit bijection between $m$-level rook placements of $k$ rooks on $B$ and $m$-level rook placements of $k$ rooks on $l_{i,p}(B)$.
\end{lem}

\begin{proof}
Use the bijection induced by the $l$ operation in Lemma~\ref{l operator} on the subboard transposed by $l_{i,p}$, not moving the rooks on the part of board $B$ which is fixed. However, this may cause a rook in the transposed subboard to occupy the same column or level of $l_{i,p}(B)$ as one of the rooks which was fixed. We deal with this possibility next.

In order for two rooks to end up in the same column, there must be rooks placed on $B$ beneath $B_{i,p}$, so we can assume $p > 1$ without losing generality. Consider the set of columns of $B$ which do not contain rooks in $B_{i,p}$, and the set of columns of $l_{i,p}(B)$ which do not contain rooks in $l(B_{i,p})$. By our assumption on $p$, these two sets have the same cardinality and so
 we can put a canonical bijection on them by pairing the leftmost columns in each set and moving to the right. If there is a rook lower than level $p$ in one of these columns of $B$, use this bijection on the columns to move it to the cell in the same row of the corresponding column of $l_{i,p}(B)$. After doing so, there must be at most one rook in each column of $l_{i,p}(B)$. For example, in Figure~\ref{permissible} the rook in $(3,2)$ is in the second column from the left of $B$ which does not contain a rook in $B_{4,2}$. Thus it moves to column 2, which is the second column from the left of $l_{4,2}(B)$ that does not contain a rook in $l(B_{4,2})$.

If two rooks end up in the same level we treat them similarly where we can assume, without loss of generality, that the $i$-th column is not the rightmost column of $B$. There is a canonical bijection between the levels of $B$ which do not contain rooks in $B_{i,p}$ and those of $l_{i,p}(B)$ that do not contain rooks in $l(B_{i,p})$. Adjust the levels of all rooks to the right of column $i$ using this bijection, fixing the column of the rook that moves. Furthermore, fix the height of the rook that moves within the level, that is, if the rook was in cell $(x,y)$, move the rook to cell $(x,y')$ in the appropriate level with $y \equiv y' \pmod{m}$. Note that since $B$ and $l_{i,p}(B)$ are singleton boards, columns to the right of column $i$ will contain a full $m$ cells at any level which contained a rook in the subboard $B_{i,p}$ or $l(B_{i,p})$.

To see that this is a bijection, we construct its inverse. Recall that the $l$ operator is an involution on singleton boards. Thus, since $B_{i,p}$ is a singleton subboard,   $l_{i,p}(l_{i,p}(B)) = B$. Similarly, applying the bijection from Lemma~\ref{l operator} and then its inverse returns the original placement of rooks on $B_{i,p}$. All that remains to check is that any rooks moved outside of $B_{i,p}$ return to their original cells. Since the rooks return to their original placement on $B_{i,p}$, the set of columns that gain a rook in $l(B_{i,p})$ after the first application of $l$ will be the same set as those that lose a rook in $B_{i,p}$ after the second application of $l$. Thus the bijection on the columns induced by the first application of $l$ will be the inverse of the bijection induced by the second application, and any rook required to move in $l_{i,p}(B)$ will move back in $l_{i,p}(l_{i,p}(B))$. A similar argument holds for levels, noting that $l_{i,p}(B)$ being singleton ensures that any level which gains a rook in $l(B_{i,p})$ after applying $l$ contains a full $m$ cells in every column to the right of column $i$. Thus this yields a bijection between rook placements on $B$ and $l_{i,p}(B)$. Figure~\ref{permissible} illustrates this bijection. 
\end{proof}

\subsection{Bijections with $m$-increasing boards}

Foata and Sch\"{u}tzenberger proved there is a unique Ferrers board in every rook equivalence class whose column lengths are strictly increasing and used this board as a target for their bijections.  To accomplish the same thing, we need the following definition and theorem.

\bfi
\begin{tikzpicture}[scale=.7]
\draw (-1,2.5) node{$B=$};
\fill[lightgray] (1,2) rectangle (4,4);
\fill[lightgray] (3,4) rectangle (4,5);
\draw[very thick] (0,0)--(4,0);
\draw (0,1)--(4,1);
\foreach \y in {2,4}
    \draw[very thick] (1,\y)--(4,\y);
\draw (1,3)--(4,3);
\draw (3,5)--(4,5);
\draw (0,0)--(0,1);
\foreach \x in {1,2}
    \draw (\x,0)--(\x,4);
\foreach \x in {3,4}
    \draw (\x,0)--(\x,5);
\draw (3.5,4.5) node{$R$};
\draw (2.5,1.5) node{$R$};
\draw (1.5,2.5) node{$R$};
\draw (3.75,4.75) node{$^1$};
\draw (3.75,2.75) node{$^1$};
\draw (3.75,3.75) node{$^2$};
\draw (2.75,2.75) node{$^3$};
\draw (2.75,3.75) node{$^4$};
\draw (1.75,2.75) node{$^5$};
\draw (1.75,3.75) node{$^6$};
\end{tikzpicture}
\hs{50pt}
\begin{tikzpicture}[scale=.7]
\draw (-1.5,2.5) node{$l_{4,2}(B)=$};
\draw[very thick] (0,0)--(4,0);
\draw (0,1)--(4,1);
\draw[very thick] (1,2)--(4,2);
\foreach \y in {5,7}
    \draw (3,\y)--(4,\y);
\draw (2,3)--(4,3);
\foreach \y in {4,6,8}
    \draw[very thick] (3,\y)--(4,\y);
\draw (0,0)--(0,1);
\draw (1,0)--(1,2);
\draw (2,0)--(2,3);
\foreach \x in {3,4}
    \draw (\x,0)--(\x,8);
\draw (2.75,2.75) node{$^1$};
\draw (3.75,2.75) node{$^1$};
\draw (3.75,3.75) node{$^2$};
\draw (3.75,4.75) node{$^3$};
\draw (3.75,5.75) node{$^4$};
\draw (3.75,6.75) node{$^5$};
\draw (3.75,7.75) node{$^6$};
\draw (2.5,2.5) node{$R$};
\draw (3.5,6.5) node{$R$};
\draw (1.5,1.5) node{$R$};
\end{tikzpicture}
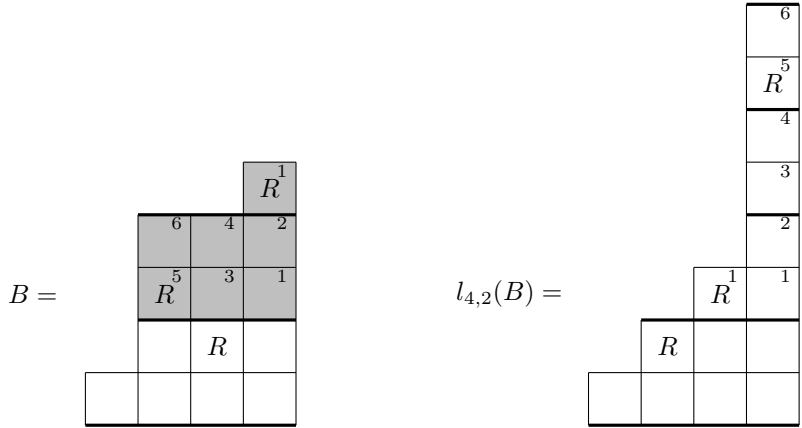
\capt{\label{permissible} On the left, $B_{4,2}$ is shaded. Here $l_{4,2}$ is permissible for $B$ and $l_{4,2}(B)$ is shown on the right.}
\efi

\begin{defn}
A Ferrers board $B=(b_1,b_2,\dots,b_n)$ is called \emph{$m$-increasing} if $b_{i+1}\ge b_i + m$ for all $1 \leq i \leq n-1$.
\end{defn}

Notice that when $m=1$ increasing and $m$-increasing are equivalent.

\begin{thm}[Theorem 4.5~\cite{blrs:mrp}]
\label{unique mrp}
Every Ferrers board is $m$-level rook equivalent to a unique $m$-increasing board.
\end{thm}

We are now almost ready to prove the main result of this section, Theorem~\ref{bijections} below. However, to do so we must put an order on Ferrers boards. Once we have established this order, we will be able to give an explicit bijection between $m$-level rook placements on an arbitrary Ferrers board $B$ and on an $m$-level rook equivalent Ferrers board which is greater than $B$ in this order, if such a board exists. Additionally, the set of all Ferrers boards equivalent to $B$ will have a unique maximum element under this order, namely the $m$-increasing board guaranteed by the previous theorem.

To  define this order, if $B=(b_1,\dots,b_n)$ then consider the reversal of $B$, $B^r=(b_n,\dots,b_1)$.  Now let $B<B'$ if   $B^r$ is lexicographically smaller than $(B')^r$. It is important to note that when applying Lemma~\ref{singleton} we will always have
\begin{equation}
\label{B_S:eq}
B_S\ge B
\end{equation}
since in $B_S$ all the cells in each level are as far to the right as possible.

\begin{lem}
\label{increase}
Given a singleton board $B$ containing a column $i$ and a level $p$ with the property that
\begin{equation}
\label{inc:eq}
\arm_m(i,p) < \leg_m(i,p),
\end{equation}
there is a singleton board $B' = l_{i',p}(B)$ with $i'\geq i$ and $B'>B$.

Furthermore, if $B$ is not $m$-increasing then a column $i$ and level $p$ satisfying equation~(\ref{inc:eq}) must exist.
\end{lem}

\begin{proof}
To prove the first statement, let $i'\ge i$ be the maximum index such that $\arm_m(i',p) < \leg_m(i',p)$.  Note that by our convention on $\arm_m$, we must have that $i'$ is at most the number of columns of $B$. We claim that it  suffices to show that $l_{i',p}$ is permissible for $B$.  This is because if  $l_{i',p}$ is permissible for $B$, then the resulting board $B'$ must satisfy $B'>B$.  Indeed, $l_{i',p}(B)$ increases the length of column $i'$ by $\leg_m(i',p)-\arm_m(i',p)$, which must be greater than $0$, and column $i'$ is the rightmost column of $B$ affected by  $l_{i',p}$. Thus $B'>B$.

If $l_{i',p}$ is not permissible for $B$, then we claim that we have $\arm_m(i'+1,p) < \leg_m(i'+1,p)$ which will contradict the maximality of $i'$ and complete this part of the proof.  Note that 
$$
\arm_m(i',p) < \leg_m(i',p) \leq \lfloor \leg_m(i',p-1) \rfloor_m.
$$
 So $l_{i',p}$ not being permissible for $B$ implies that $\lfloor \arm_m(i'+1,p) \rfloor_m<\leg_m(i',p) =\leg_m(i'+1,p) -m$  since $B$ is singleton and, because  $\leg_m(i',p)$ is positive,  $i'$ cannot be the leftmost column terminating in level $p$.  This implies the desired contradiction that $\arm_m(i'+1,p) < \leg_m(i'+1,p)$. 

To prove the second statement of the theorem, note that if $B$ is not $m$-increasing there are two possible cases: either there are two adjacent columns $i-1, i$ of $B$ 
which terminate at the same level, or column $i-1$ terminates in level $p$ and $B$ has exactly $r_1$ cells in the $p$th level of column $i-1$ and exactly $r_2$ cells in the $(p+1)$st  level of column $i$ where $r_1 > r_2 > 0$.

Case 1: Let columns $i-1$ and $i$ both terminate at level $p$. Then $\arm_m(i,p) = 0$, by the assumption that column $i$ terminates at level $p$, but $\leg_m(i,p) \geq 1$ since column $i-1$ also terminates at the $p$th level. Thus $\arm_m(i,p) < \leg_m(i,p)$ as desired.

Case 2: By assumption $\arm_m(i,p) = r_2 < r_1 \leq \leg_m(i,p)$ which completes the proof.
\end{proof}

We are now in a position to prove our main theorem of this section.

\begin{thm}
\label{bijections}
Given any two $m$-level rook equivalent Ferrers boards, there is an explicit bijection between $m$-level rook placements of $k$ rooks on them.
\end{thm}

\begin{figure} \begin{center}
\begin{tikzpicture}[scale=.5]
\foreach \y in {0,1}
	\draw (0,\y)--(5,\y);
\foreach \y in {2,3,5}
	\draw (3,\y)--(5,\y);
\foreach \y in {4,5}
	\draw (4,\y)--(5,\y);
\draw[very thick] (3,6)--(5,6);
\foreach \x in {0,1,2}
	\draw (\x,0)--(\x,1);
\draw (3,0)--(3,6);
\foreach \x in {4,5}
	\draw (\x,0)--(\x,7);
\draw (4,7)--(5,7);
\draw (.5,.5) node{$R$};
\draw (4.5,2.5) node{$R$};
\draw (3.5,5.5) node{$R$};
\draw (-1,2.5) node{(a)};
\draw[very thick] (0,0)--(5,0);
\draw[very thick] (3,2)--(5,2);
\draw[very thick] (3,4)--(5,4);
\end{tikzpicture}
\hs{40pt}
\begin{tikzpicture}[scale=.5]
\foreach \y in {0,1}
	\draw (0,\y)--(4,\y);
\draw[very thick] (1,2)--(4,2);
\draw (2,3)--(4,3);
\foreach \y in {4,5,}
	\draw (2,\y)--(4,\y);
\draw[very thick] (2,6)--(4,6);
\draw[very thick] (0,0)--(4,0);
\draw[very thick] (2,4)--(4,4);
\foreach \x in {0,1}
	\draw (\x,0)--(\x,\x+1);
\draw (2,0)--(2,6);
\foreach \x in {3,4}
	\draw (\x,0)--(\x,7);
\draw (3,7)--(4,7);
\draw (.5,.5) node{$R$};
\draw (3.5,2.5) node{$R$};
\draw (2.5,5.5) node{$R$};
\draw (-1,2.5) node{(b)};
\end{tikzpicture}
\hspace{40pt}
\begin{tikzpicture}[scale=.5]
\foreach \y in {0,1}
	\draw (0,\y)--(4,\y);
\draw[very thick] (1,2)--(4,2);
\draw (2,3)--(4,3);
\foreach \y in {4,5,}
	\draw (2,\y)--(4,\y);
\draw[very thick] (3,6)--(4,6);
\draw[very thick] (0,0)--(4,0);
\draw[very thick] (2,4)--(4,4);
\foreach \x in {0,1}
	\draw (\x,0)--(\x,\x+1);
\draw (2,0)--(2,5);
\foreach \x in {3,4}
	\draw (\x,0)--(\x,8);
\draw (3,7)--(4,7);
\draw[very thick] (3,8)--(4,8);
\draw (.5,.5) node{$R$};
\draw (2.5,2.5) node{$R$};
\draw (3.5,7.5) node{$R$};
\draw (-1,2.5) node{(c)};
\end{tikzpicture}
\hspace{40pt}
\begin{tikzpicture}[scale=.5]
\foreach \y in {0,1,2,3}
	\draw (0,\y)--(3,\y);
\foreach \y in {0,2}
	\draw[very thick] (0,\y)--(3,\y);
\foreach \y in {6,8}
	\draw[very thick] (2,\y)--(3,\y);
\foreach \y in {7}
	\draw(2,\y)--(3,\y);
\draw[very thick] (1,4)--(3,4);
\draw[very thick] (1,2)--(3,2);
\draw (1,3)--(3,3);
\draw (1,5)--(3,5);
\draw (0,0)--(0,3);
\draw (1,0)--(1,5);
\foreach \x in {2,3}
	\draw (\x,0)--(\x,8);
\draw (.5,2.5) node{$R$};
\draw (1.5,.5) node{$R$};
\draw (2.5,7.5) node{$R$};
\draw (-1,2.5) node{(d)};
\end{tikzpicture}

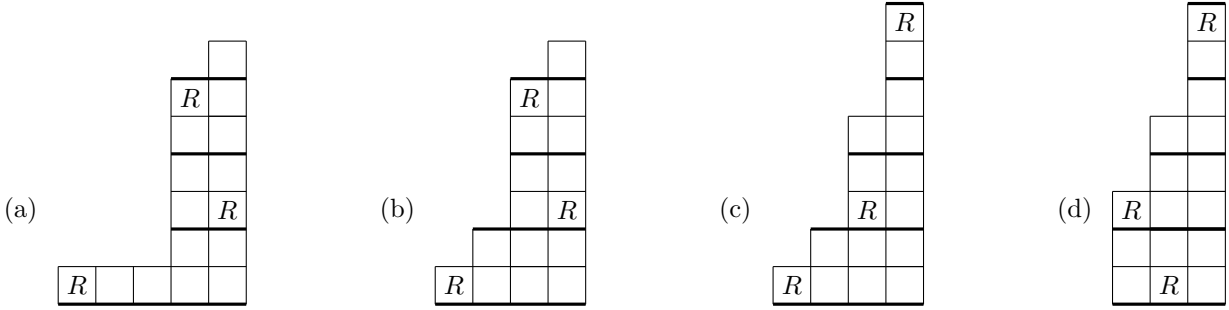
\capt{\label{full bijection} (a) A $2$-level rook placement on a Ferrers board.	(b) The  placement on the singleton board obtained after applying  Lemma~\ref{singleton}.	(c) The  placement obtained after applying Lemma~\ref{local l} using $l_{4,3}$.
(d) The placement obtained on a $2$-increasing board after applying Lemma~\ref{local l} again using $l_{2,1}$.}
\efi

\begin{proof}
Given any Ferrers board $B$, let $B_m$ be the unique $m$-increasing board  in the $m$-level rook equivalence class of $B$ guaranteed by Theorem~\ref{unique mrp}.
It suffices to show that there is an explicit bijection between the $m$-level rook placements of $k$ rooks on $B$ and those on $B_m$.  This is trivial if $B=B_m$ so assume $B\neq B_m$.  By  Lemma~\ref{singleton}, we have an explicit bijection between the placements on $B$ and those on $B_S$ where $B_S\ge B$ by equation~(\ref{B_S:eq}).  If $B_S=B_m$ then we are done.  Otherwise, apply the local $l$ operator defined in Lemma~\ref{increase} which will give $B'=l_{i,p}(B_S)$ with $B'>B_S$ and, by Lemma~\ref{local l}, another explicit bijection between rook placements.  We now repeat this process if necessary.  Since there are only finitely many boards in an $m$-level rook equivalence class and the lexicographic order increases at each stage, we must eventually terminate.  And, by Lemma~\ref{increase} again, termination must occur at $B_m$.  Composing all the bijections finishes the proof. 
\end{proof}

See Figure~\ref{full bijection} for a short example of this process.

\section{$q$-Analogues}
\label{sec:q-analogues}

Briggs and Remmel~\cite{br:mrn} defined $p,q$-analogues of
the $m$-level rook numbers, denoted $r_{k,m}(B;p,q)$, by assigning
a monomial in $p$ and $q$ to each $m$-level rook placement of $k$
rooks on $B$.  Briggs and Remmel proved a factorization formula
involving $r_{k,m}(B;p,q)$ for singleton boards, which was generalized
to all Ferrers boards by the present authors~\cite[Thm. 3.3]{blrs:mrp}. 
In this section, we show that the bijections given earlier in this paper
preserve the $q$-power assigned to a rook
placement. This leads to bijective proofs that two $m$-level rook equivalent boards have the
same rook polynomials $r_{k,m}(B;1,q)$ for all $k$.  Our bijections do
not preserve the $p$-power, however, and we leave it as an open problem
to give a bijective treatment of the full $p,q$-analogue of $m$-level
rook numbers.

\subsection{Definition of the $q$-weight}
\label{subsec:def-qwt}

To begin, we recall that the $q$-weight assigned to
an $m$-level rook placement $\pi$ on a board $B$ is the $m$-inversion number of $\pi$. 
The \emph{$m$-inversion number}, denoted $\inv_m(\pi)$, counts cells $c$ in $B$ satisfying the following
conditions:
\begin{enumerate}

\item The cell $c$ does not contain a rook.
\item There is no rook above $c$
in the same column.
\item There is no rook to the left of $c$ in the same level.

\end {enumerate}

For example, the $3$-level rook placement shown in Figure~\ref{q weight example} has an $m$-inversion number of $19$;  
the cells contributing to the $m$-inversion number are marked by stars.

To motivate why this statistic is called the $m$-inversion number, consider the case where $m=1$. Thus the $1$-inversion number counts the number of cells which do not contain a rook and are neither below nor to the right of a rook. If $B$ is an $n$ by $n$ board and $\sigma$ is an element of $S_n$, the symmetric group on the elements $\{1,2,\dots,n\}$, then we can associate with $\sigma$ a placement of $n$ rooks on $B$, $\pi$, by the convention that there is a rook in column $i$ and row $n+1-p$ if and only if $\sigma_i=p$. In this case $\inv_1(\pi) = \inv(\sigma)$ where $\inv(\sigma)$ is the standard inversion number of a permutation, counting the number of pairs of indices $(a,b)$ with the property that $a<b$ but $\sigma(a)>\sigma(b)$.

\bfi
\begin{tikzpicture}[scale=.5]
\foreach \y in {1,2]}
    \draw (0,\y)--(8,\y);
\draw[very thick] (0,0)--(8,0);
\draw[very thick] (1,3)--(8,3);
\draw (3,4)--(8,4);
\draw (4,5)--(8,5);
\draw[very thick] (4,6)--(8,6);
\draw (4,7)--(8,7);
\draw (5,8)--(8,8);
\draw[very thick] (6,9)--(8,9);
\draw (6,10)--(8,10);
\draw (0,0)--(0,2);
\foreach \x in {1,2} 
   \draw (\x,0)--(\x,3);
\draw (3,0)--(3,4);
\draw (4,0)--(4,7);
\draw (5,0)--(5,8);
\foreach \x in {6,7,8}
	\draw (\x,0)--(\x,10);
\draw (2.5,1.5) node{$R$};
\draw (5.5, 6.5) node{$*$};
\draw (7.5,5.5) node{$R$};
\draw (.5,.5) node{$*$};
\draw (1.5,.5) node{$*$};
\draw (.5,1.5) node{$*$};
\draw (1.5,1.5) node{$*$};
\draw (5.5,3.5) node{$*$};
\draw (5.5,4.5) node{$*$};
\draw (5.5,5.5) node{$*$};
\draw (1.5,2.5) node{$*$};
\draw (2.5,2.5) node{$*$};
\draw (3.5,3.5) node{$*$};
\draw (4.5,3.5) node{$*$};
\draw (4.5,4.5) node{$*$};
\draw (4.5,5.5) node{$*$};
\draw (4.5,6.5) node{$*$};
\draw (5.5,7.5) node{$*$};
\draw (7.5,6.5) node{$*$};
\draw (7.5,7.5) node{$*$};
\draw (7.5,8.5) node{$*$};
\draw (6.5,9.5) node{$R$};
\end{tikzpicture}
\capt{\label{q weight example} A placement, $\pi$, with $\inv_3(\pi)=19$.}
\efi

Define $r_{k,m}(B;q)$ by
\begin{equation}\label{q rook poly}
r_{k,m}(B;q) = \sum_{\pi} q^{\inv_m(\pi)}
\end{equation}
where $\pi$ ranges over all $m$-level rook placements of $k$ rooks on $B$.
Define boards $B$ and $B'$ to be 
\emph{$m$-level $q$-rook equivalent} if $r_{k,m}(B;q)=r_{k,m}(B';q)$
for all nonnegative integers $k$. We will give bijective proofs of
the $m$-level $q$-rook equivalence of various boards, by showing that
the bijections given earlier preserve the $q$-power.

We will need two formulas for the $m$-inversion number of an $m$-level rook placement
$\pi$, one which adds the contributions of each level, and 
another which sums over the contributions of each column.
The first formula uses the numbering of cells in each level
from Definition~\ref{level numbering}. For each level $p$ such that $\pi$ has a rook $R$ 
in level $p$, let $h_p(\pi)$ count the cells in level $p$ with a higher 
number than the cell containing $R$. Also let $\NW_p(\pi)$ be the number of
rooks in $\pi$ northwest of $R$, that is, rooks in a higher level and earlier
column than $R$.  For each level $p$ containing no rook,
let $h_p(\pi)$ be the total number of cells in this level,
and let $\NW_p(\pi)$ be the number of rooks in higher levels than $p$. 
Note that the definition for a level containing no rooks can be considered as a limiting case of the one for a level containing a rook  by letting the rook move to the right until it exits the board.  So in our proofs we will only consider the first case as the second one will automatically follow using this procedure.
Define, using ``h" for ``horizontal,"
\begin{equation}\label{eq:qwt-lvl}
 \hinv_p(\pi)=h_p(\pi)-m\cdot\NW_p(\pi).
\end{equation}
It is routine to check that for any board $B$,
$\inv_m(\pi)=\sum_{p\geq 1} \hinv_p(\pi)$.
In particular, if a column has fewer than $m$ cells in level $p$,
there can be no rook weakly west of this column in a higher level. Thus each 
rook counted by $\NW_p(\pi)$ removes a full $m$ cells from the cells that
would have contributed to $\inv_m(\pi)$ in level $p$. In the example in Figure~\ref{q weight example},
$\hinv_1(\pi)=6$, $\hinv_2(\pi)=7$, $\hinv_3(\pi)=6$, and $\hinv_4(\pi)=0$.

\bfi
\begin{tikzpicture}[scale=.5]
\foreach \y in {0,3}
	\draw[very thick] (0,\y)--(7,\y);
\foreach \y in {1,2,4}
	\draw (0,\y)--(7,\y);
\foreach \y in {5,7}
	\draw (3,\y)--(7,\y);
\draw[very thick] (3,6)--(7,6);
\foreach \y in {8,10}
	\draw (4,\y)--(7,\y);
\draw[very thick] (4,9)--(7,9);
\foreach \x in {0,1,2}
	\draw (\x,0)--(\x,4);
\draw (3,0)--(3,7);
\foreach \x in {4,5,6,7}
	\draw (\x,0)--(\x,10);
\draw (.5,3.5) node{$R$};
\draw (2.5,1.5) node{$R$};
\draw (4.5,9.5) node{$R$};
\draw (5.5,7.5) node{$R$};
\draw (1.5,.5) node{$*$};
\draw (1.5,1.5) node{$*$};
\draw (1.5,2.5) node{$*$};
\draw (2.5,2.5) node{$*$};
\draw (3.5,6.5) node{$*$};
\draw (5.5,8.5) node{$*$};
\end{tikzpicture}
\hs{50pt}
\begin{tikzpicture}[scale=.5]
\foreach \y in {0,3}
	\draw[very thick] (0,\y)--(7,\y);
\foreach \y in {1,2}
	\draw (0,\y)--(7,\y);
\foreach \y in {4,5}
	\draw (2,\y)--(7,\y);
\draw[very thick] (2,6)--(7,6);
\draw (3,7)--(7,7);
\draw (4,8)--(7,8);
\draw[very thick] (4,9)--(7,9);
\foreach \y in {10,11}
	\draw (6,\y)--(7,\y);
\draw[very thick] (6,12)--(7,12);
\foreach \x in {0,1}
	\draw (\x,0)--(\x,3);
\draw (2,0)--(2,6);
\draw (3,0)--(3,7);
\foreach \x in {4,5}
	\draw (\x,0)--(\x,9);
\foreach \x in {6,7}
	\draw (\x,0)--(\x,12);
\draw (1.5,1.5) node{$R$};
\draw (2.5,5.5) node{$R$};
\draw (4.5,7.5) node{$R$};
\draw (6.5,11.5) node{$R$};
\draw (.5,.5) node{$*$};
\draw (.5,1.5) node{$*$};
\draw (.5,2.5) node{$*$};
\draw (1.5,2.5) node{$*$};
\draw (3.5,6.5) node{$*$};
\draw (4.5,8.5) node{$*$};
\end{tikzpicture}
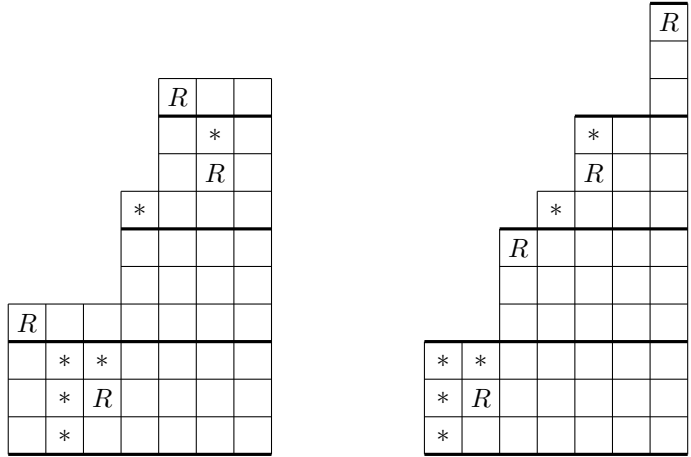
\capt{\label{q weight singleton} On the left, a placement $\pi$ with $\inv_3\pi=6$ on a Ferrers board. On the right, the corresponding placement on $B_S$.}
\efi

The second formula for $\inv_m(\pi)$ classifies cells based on their columns.
For each column $i$ such that $\pi$ has a rook $R$ in column $i$,
let $h_i'(\pi)$ count the cells in column $i$ above $R$, and let
$\NW_i'(\pi)$ be the number of rooks in $\pi$ northwest of $R$.
For each column $i$ containing no rook, let $h_i'(\pi)$ be the total
number of cells in this column, and let $\NW_i'(\pi)$ be the
number of rooks in earlier columns than $i$. 
Again, the second case is a limiting instance of the first where now the rook moves down until it is off the board.
Define, using ``v" for ``vertical,"
\begin{equation}\label{eq:qwt-col}
 \vinv_i(\pi)=h_i'(\pi)-m\cdot\NW_i'(\pi).
\end{equation}
One may check that for any rook placement $\pi$ on a singleton
board
$B$, $\inv_m(\pi)=\sum_{i\geq 1} \vinv_i(\pi)$. The singleton condition ensures
that any rook counted by $\NW_i'(\pi)$ must remove a full $m$ cells from
the cells that would have contributed to $\inv_m(\pi)$ in column $i$.
In the example from Figure~\ref{q weight example}, it happens that $\inv_m(\pi)=19$ is not the sum of the
entries in $(\vinv_1(\pi),\ldots,\vinv_8(\pi))=(2,3,1,1,4,5,0,1)$
because $B$ is not a singleton board and the rook to the northwest of the rook in the rightmost column only cancels one cell in the rightmost column, rather than a full $3$.

\subsection{Mapping placements on $B$ to placements on $B_S$}
\label{subsec:map-B-to-BS}

We now prove that the bijection in Lemma~\ref{singleton},
mapping $m$-level rook placements on an arbitrary board $B$
to $m$-level rook placements on the singleton board $B_S$, 
preserves the $m$-inversion number. See Figure~\ref{q weight singleton} for an example in the case $m=3$.

\begin{lem}
If $\pi$ is a rook placement on a Ferrers board $B$ that maps to 
the rook placement $\pi_S$ on $B_S$ when we apply the bijection in Lemma~\ref{singleton}, then $\inv_m(\pi)=\inv_m(\pi_S)$.
\end{lem}

\begin{proof}
We use~\eqref{eq:qwt-lvl} to show that $\hinv_p(\pi_S)=\hinv_p(\pi)$ for
each level $p$. Let $\pi'$ 
be the placement created from $\pi$ in the
first stage of the map, in which all rooks remain in their original
numbered cell in their level. By definition of the level numbering, 
$h_p(\pi')=h_p(\pi)$ for all $p$. 
Consider a rook that moves from column $i$ to
column $i'>i$ in the first stage, 
and a level $p$ below that rook
that has a rook in the interval $(i,i']$. In such a level,
$\NW_p(\pi')=\NW_p(\pi)-1$, so $\hinv_m(\pi')=\hinv_m(\pi)+m$. The second stage
corrects for this increase by moving the rook in level $p$ one
column to the left, which decreases $h_p$ by $m$. The net effect
is that $\hinv_p(\pi_S)=\hinv_p(\pi)$ and hence $\inv_m(\pi_S)=\inv_m(\pi)$, as needed.
\end{proof}

\subsection{Analysis of the $l$-operator}
\label{subsec:analyze-l-op}

Let $B$ be a singleton board. We now show that the bijection from 
Lemma~\ref{l operator}, which maps an $m$-level rook placement $\pi$ on $B$ to an
$m$-level rook placement $l(\pi)$ on $l(B)$, preserves the $m$-inversion number.

\begin{lem}
If $\pi$ is an $m$-level rook placement on a singleton board $B$, then $\inv_m(\pi)=\inv_m(l(\pi))$.
\end{lem}

\begin{proof}
The level numbering of $B$ and the column numbering of $l(B)$ induce a bijection between the squares of $B$ and the squares of $l(B)$.  It is easy to see from the definitions that a square of $B$ contributes to $\inv_m(\pi)$ if and only if the corresponding square of $l(B)$ contributes to $\inv_m(l(\pi))$.  So the lemma is proved.
\end{proof}

The example in Figure~\ref{q weight l} illustrates the ideas in this proof in a case where $m=2$;
 note that the starred cells in each level of the original 
 placement become starred cells in each column of the new placement.

\bfi
\begin{tikzpicture}[scale=.5]
\draw[very thick] (0,0)--(6,0);
\draw (0,1)--(6,1);
\foreach \y in {2,4}
	\draw[very thick] (1,\y)--(6,\y);
\foreach \y in {3,5}
	\draw (1,\y)--(6,\y);
\draw[very thick] (2,6)--(6,6);
\draw (2,7)--(6,7);
\draw[very thick] (3,8)--(6,8);
\draw (0,0)--(0,1);
\draw (1,0)--(1,5);
\draw (2,0)--(2,7);
\foreach \x in {3,4,5,6}
	\draw (\x,0)--(\x,8);
\draw (1.5,1.5) node{$R$};
\draw (2.5,6.5) node{$R$};
\draw (3.5,4.5) node{$R$};
\draw (5.5,3.5) node{$R$};
\draw (.5,.5) node{$*$};
\draw (1.5,2.5) node{$*$};
\draw (1.5,3.5) node{$*$};
\draw (1.5,4.5) node{$*$};
\draw (3.5,5.5) node{$*$};
\draw (4.5,2.5) node{$*$};
\draw (4.5,3.5) node{$*$};
\end{tikzpicture}
\hs{50pt}
\begin{tikzpicture}[scale=.5]
\foreach \y in {0,2,4,6}
	\draw[very thick] (0,\y)--(4,\y);
\foreach \y in {1,3,5,7}
	\draw (0,\y)--(4,\y);
\draw[very thick] (1,8)--(4,8);
\draw (1,9)--(4,9);
\draw[very thick] (2,10)--(4,10);
\draw (3,11)--(4,11);
\draw (0,0)--(0,7);
\draw (1,0)--(1,9);
\draw (2,0)--(2,10);
\draw (3,0)--(3,11);
\draw (4,0)--(4,11);
\draw (.5,6.5) node{$R$};
\draw (1.5,4.5) node{$R$};
\draw (2.5,1.5) node{$R$};
\draw (3.5,9.5) node{$R$};
\draw (1.5,5.5) node{$*$};
\draw (1.5,8.5) node{$*$};
\draw (2.5,2.5) node{$*$};
\draw (2.5,3.5) node{$*$};
\draw (2.5,8.5) node{$*$};
\draw (2.5,9.5) node{$*$};
\draw (3.5,10.5) node{$*$};
\end{tikzpicture}
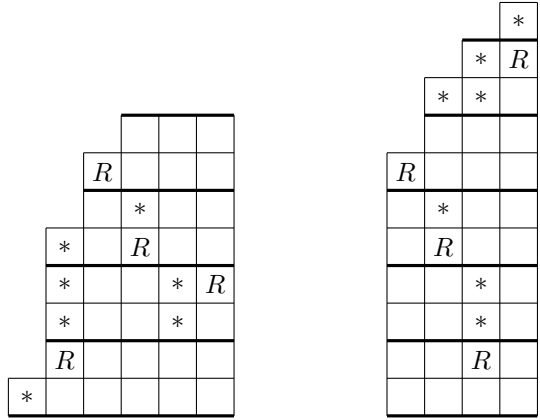
\capt{\label{q weight l} On the left, a $2$-level placement with a $2$-inversion number of $7$ on a singleton board. On the right, the corresponding placement after applying the $l$-operator to the board on the left.}
\efi

\subsection{The local $l$-operator}
\label{subsec:local-l-op}

Next we show that the bijection in Lemma~\ref{local l} preserves the $m$-inversion number.

\begin{lem}
Given an $m$-level rook placement $\pi$ on a singleton board $B$, a column $i$, and a level $p$ such that
$l_{i,p}$ is permissible for $B$, let $\pi'$ denote the corresponding placement on $l_{i,p}(B)$ as described in Lemma~\ref{local l}.  Then $\inv_m(\pi)=\inv_m(\pi')$.
\end{lem}

\bfi
\begin{tikzpicture}[scale=.5]
\foreach \y in {0,2}
	\draw[very thick] (0,\y)--(6,\y);
\foreach \y in {1,3}
	\draw (0,\y)--(6,\y);
\foreach \y in {4,6}
	\draw[very thick] (1,\y)--(6,\y);
\foreach \y in {5,7}
	\draw (1,\y)--(6,\y);
\draw[very thick] (2,8)--(6,8);
\draw (4,9)--(6,9);
\draw[very thick] (4,10)--(6,10);
\draw (0,0)--(0,3);
\draw (1,0)--(1,7);
\foreach \x in {2,3}
	\draw (\x,0)--(\x,8);
\foreach \x in {4,5,6}
	\draw (\x,0)--(\x,10);
\draw (.5,2.5) node{$R$};
\draw (2.5,6.5) node{$R$};
\draw (3.5,1.5) node{$R$};
\draw (4.5,4.5) node{$R$};
\draw (5.5,8.5) node{$R$};
\draw (1.5,.5) node{$*$};
\draw (1.5,1.5) node{$*$};
\draw (1.5,4.5) node{$*$};
\draw (1.5,5.5) node{$*$};
\draw (1.5,6.5) node{$*$};
\draw (2.5,7.5) node{$*$};
\draw (3.5,4.5) node{$*$};
\draw (3.5,5.5) node{$*$};
\draw (4.5,5.5) node{$*$};
\draw (4.5,8.5) node{$*$};
\draw (4.5,9.5) node{$*$};
\draw (5.5,9.5) node{$*$};
\end{tikzpicture}
\hs{50pt}
\begin{tikzpicture}[scale=.5]
\foreach \y in {0,2}
	\draw[very thick] (0,\y)--(6,\y);
\draw (0,1)--(6,1);
\foreach \y in {3,5,7}
	\draw (1,\y)--(6,\y);
\foreach \y in {4,6}
	\draw[very thick] (1,\y)--(6,\y);
\draw[very thick] (2,8)--(6,8);
\draw (3,9)--(6,9);
\draw[very thick] (4,10)--(6,10);
\draw (0,0)--(0,2);
\draw (1,0)--(1,7);
\draw (2,0)--(2,8);
\draw (3,0)--(3,9);
\foreach \x in {4,5,6}
	\draw (\x,0)--(\x,10);
\draw (1.5,4.5) node{$R$};
\draw (2.5,1.5) node{$R$};
\draw (3.5,8.5) node{$R$};
\draw (4.5,2.5) node{$R$};
\draw (5.5,6.5) node{$R$};
\draw (.5,.5) node{$*$};
\draw (.5,1.5) node{$*$};
\draw (1.5,5.5) node{$*$};
\draw (1.5,6.5) node{$*$};
\draw (2.5,2.5) node{$*$};
\draw (2.5,3.5) node{$*$};
\draw (2.5,6.5) node{$*$};
\draw (2.5,7.5) node{$*$};
\draw (4.5,3.5) node{$*$};
\draw (4.5,6.5) node{$*$};
\draw (4.5,7.5) node{$*$};
\draw (5.5,7.5) node{$*$};
\end{tikzpicture}
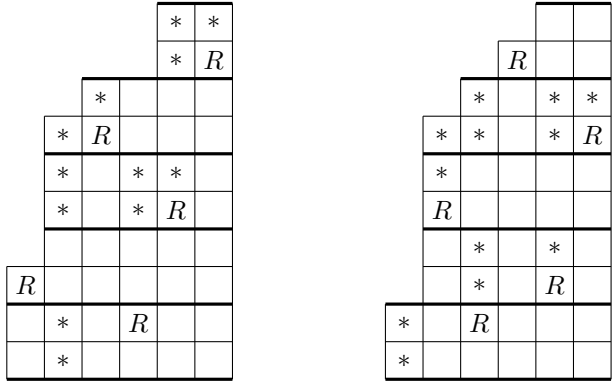
\capt{\label{q weight local l} On the left, a placement with a $2$-inversion number of $12$ on a singleton board. On the right, the corresponding placement after applying the local $l$-operator $l_{4,2}$ to the board on the left.}
\efi

\begin{proof}
We have already shown in \S\ref{subsec:analyze-l-op}
that the contribution to the $m$-inversion number coming from
the cells in $B_{i,p}$ and $l(B_{i,p})$ is the same.  We show that the adjustments made below $B_{i,p}$ do 
not affect the $m$-inversion number, as follows.
By the construction of the bijection in Lemma~\ref{local l}, every rook beneath $B_{i,p}$ is adjusted horizontally until there are as many columns to the left of it in $l_{i,p}(B)$ that do not contain rooks in $l(B_{i,p})$ as there  were before applying the $l$-operator. Since the relative order of rooks in levels beneath level $p$ is preserved, each level beneath level $p$ will have the same number of cells counted by the $m$-inversion number before and after the adjustment. 
An analogous argument shows that the level
adjustments to the right of $B_{i,p}$ do not change the $m$-inversion number, since after the adjustment each rook to the right of $B_{i,p}$ will have the same number of levels which get counted for the $m$-inversion number above it in $l_{i,p}(B)$ as it does in $B$.
\end{proof}

See Figure~\ref{q weight local l} 
for an example with $m=2$. Note that the rook in $(4,2)$ 
on the left moves to $(3,2)$ keeping one column with no higher rook  to the left of it, so the first level contributes $2$ to the $m$-inversion number in both placements. Also consider the rook in $(5,5)$ on the left which moves to $(5,3)$ maintaining its position in the bottom row of its level. Even though there are more rooks to the northwest of it in the right-hand diagram, the number of levels above it that do not contain rooks to the left of it is unchanged, so the fifth column contributes $3$ to the $m$-inversion number.

\begin{thm}
If $B$ and $B'$ are $m$-level rook equivalent Ferrers boards, then $r_{k,m}(B;q) = r_{k,m}(B';q)$.
\end{thm}

\begin{proof}
We can compose all the bijections, as in the proof of Theorem~\ref{bijections}, to obtain an $m$-inversion-preserving bijection between $m$-level rook placements of $k$ rooks on $B$ and $B'$. By the definition of $r_{k,m}(B;q)$ in~\S\ref{q rook poly},
this shows that $r_{k,m}(B;q) = r_{k,m}(B';q)$.
\end{proof}

\section{A second bijection on $m$-level rook placements}
\label{sec:LR placement}
Our next two main results will require the Garsia-Milne Involution Principle. First, we will use the Involution Principle to construct another explicit bijection between two arbitrary $m$-level rook placements of $k$ rooks on $m$-level rook equivalent singleton boards.

\begin{thm}[Garsia-Milne Involution Principle~\cite{gm:mcb}]
\label{garsia milne}
Consider a triple $(S,T,I)$  where $S$ is a signed set, $I$ is a sign-reversing involution on $S$, and the set $T$ of fixed points of $I$ is required to be a subset of the positive part $S^+$ of $S$.   Let $(S',T',I')$ be defined similarly. Then, given an explicit sign-preserving bijection $f$ from $S$ to $S'$, one can construct an explicit bijection between $T$ and $T'$.
\end{thm}

The way that Garsia and Milne define the explicit bijection is as follows.  Start with an element $t \in T \subseteq S^+$. If $f(t) \not\in T'$, then apply $(f \circ I \circ f^{-1} \circ I')$ to $f(t)$. This takes $f(t) \in S'^+$ to $S'^-$, then to $S^-$, then to $S^+$, and finally back to $S'^+$.  Iterating this procedure must ultimately yield an element of $T'$ which is considered the image of $t$ under the desired bijection.

\subsection{A Garsia-Milne bijection for rook placements}
\label{subsec:GMrook}

We will use the Involution Principle to construct a bijection between $m$-level rook placements on two $m$-level rook equivalent singleton boards.  We must first construct a signed set and a sign-reversing involution so that the $m$-level rook placements are the fixed points under the involution. We do this as follows.

Given two Ferrers boards, $B$ and $B'$, we shall say $B$ \emph{fits inside} $B'$ if juxtaposing the two boards with their lower right cells in the same position makes the cells of $B$ a subset of the cells of $B'$. Figure~\ref{placement 1} shows that the thick bordered $B=(2,3)$ fits inside  $B'=(0,2,4,6)$. The shading and rook placement may be ignored for now. Let $\Delta_{n,m}$ denote the triangular Ferrers board $(0,m,2m,\dots,(n-1)m)$. Given a singleton board $B$,   fix $N$ large enough that $B$ fits inside $\Delta_{N,m}$. If $B$ has fewer than $N$ columns, expand $B$ on the left with columns of height zero so $B=(b_0,b_1,\dots,b_{N-1})$ has the same number of columns as $\Delta_{N,m}$. Fix a non-negative integer $k$ with $k < N$ and let the integer $i$ vary over $0 \leq i \leq k$. Then $S$ will consist of all configurations $C$ constructed as follows.  Take  $\Delta_{N,m}$ with $B$ fitting inside and place white rooks $W$ in $i$ cells of $\Delta_{N,m}$ that are outside of $B$ so that no two white rooks are in the same column. Next, place $k-i$ black rooks $R$  forming an $m$-level rook placement on the subboard $\Delta_{N-i,m}$ which is located in the columns of $\Delta_{N,m}$ which do not contain a white rook. We will call this the \emph{inset $\Delta_{N-i,m}$ board}.  Note that the columns of the inset $\Delta_{N-i,m}$ may not be contiguous.

\bfi
\begin{tikzpicture}[scale=.5]
\fill[lightgray] (0,0) rectangle (1,2);
\fill[lightgray] (2,0) rectangle (3,4);
\draw (-1,0)--(3,0);
\foreach \y in {1,2,3,4}
    \draw (\y-1,2*\y)--(3,2*\y);
\foreach \y in {1,2,3} 
   \draw (\y-1,2*\y-1)--(3,2*\y-1);
\foreach \x in {0,1,2}
   \draw (\x,0)--(\x,2*\x+2);
\draw (3,0)--(3,6);
\draw[very thick] (1,0)--(3,0);
\draw[very thick] (1,2)--(2,2);
\draw[very thick] (2,3)--(3,3);
\draw[very thick] (1,0)--(1,2);
\draw[very thick] (2,2)--(2,3);
\draw[very thick] (3,0)--(3,3);
\draw (2.5,2.5) node{$R$};
\draw (.5, .5) node{$R$};
\draw (1.5,3.5) node{$W$};
\draw (-1,3) node{$C=$};
\draw (4,3) node{$\in S$};
\end{tikzpicture}
\hs{50pt}
\begin{tikzpicture}[scale=.5]
\fill[lightgray] (2,0) rectangle (3,2);
\draw (-1,0)--(3,0);
\foreach \y in {1,2,3,4}
    \draw (\y-1,2*\y)--(3,2*\y);
\foreach \y in {1,2,3} 
   \draw (\y-1,2*\y-1)--(3,2*\y-1);
\foreach \x in {0,1,2}
   \draw (\x,0)--(\x,2*\x+2);
\draw (3,0)--(3,6);
\draw[very thick] (1,0)--(3,0);
\draw[very thick] (1,2)--(2,2);
\draw[very thick] (2,3)--(3,3);
\draw[very thick] (1,0)--(1,2);
\draw[very thick] (2,2)--(2,3);
\draw[very thick] (3,0)--(3,3);
\draw (2.5,.5) node{$R$};
\draw (.5, .5) node{$W$};
\draw (1.5,3.5) node{$W$};
\draw (-1.5,3) node{$I(C)=$};
\end{tikzpicture}

\begin{tikzpicture}[scale=.5]
\fill[lightgray] (0,0) rectangle (1,2);
\fill[lightgray] (1,0) rectangle (2,4);
\draw (-1,0)--(3,0);
\foreach \y in {1,2,3,4}
    \draw (\y-1,2*\y)--(3,2*\y);
\foreach \y in {1,2,3} 
   \draw (\y-1,2*\y-1)--(3,2*\y-1);
\foreach \x in {0,1,2}
   \draw (\x,0)--(\x,2*\x+2);
\draw (3,0)--(3,6);
\draw[very thick] (1,0)--(3,0);
\draw[very thick] (1,1)--(2,1);
\draw[very thick] (2,4)--(3,4);
\draw[very thick] (1,0)--(1,1);
\draw[very thick] (2,1)--(2,4);
\draw[very thick] (3,0)--(3,4);
\draw (1.5,2.5) node{$R$};
\draw (.5, .5) node{$R$};
\draw (2.5,5.5) node{$W$};
\draw (-1.5,3) node{$f(C)=$};
\end{tikzpicture}
\hs{55pt}
\begin{tikzpicture}[scale=.5]
\fill[lightgray] (1,0) rectangle (2,2);
\draw (-1,0)--(3,0);
\foreach \y in {1,2,3,4}
    \draw (\y-1,2*\y)--(3,2*\y);
\foreach \y in {1,2,3} 
   \draw (\y-1,2*\y-1)--(3,2*\y-1);
\foreach \x in {0,1,2}
   \draw (\x,0)--(\x,2*\x+2);
\draw (3,0)--(3,6);
\draw[very thick] (1,0)--(3,0);
\draw[very thick] (1,1)--(2,1);
\draw[very thick] (2,4)--(3,4);
\draw[very thick] (1,0)--(1,1);
\draw[very thick] (2,1)--(2,4);
\draw[very thick] (3,0)--(3,4);
\draw (1.5,.5) node{$R$};
\draw (.5, .5) node{$W$};
\draw (2.5,5.5) node{$W$};
\draw (-1.5,3) node{$f(I(C))=$};
\end{tikzpicture}
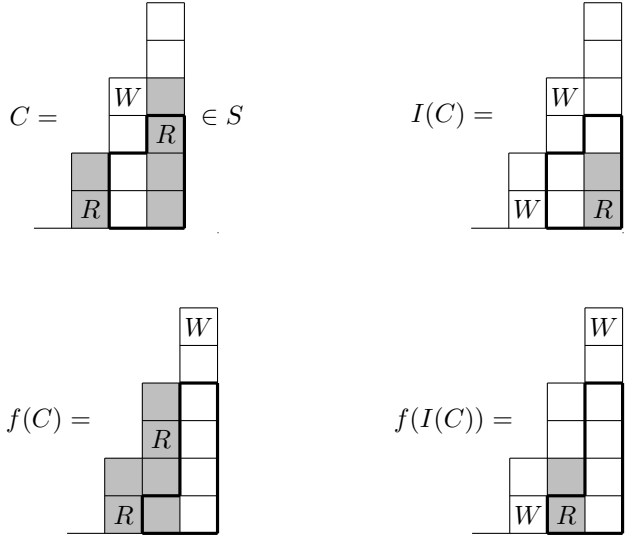
\capt{\label{placement 1} On the top left, an element in $S$ with sign $-1$. On the top right, the image under $I$ which has sign $+1$. Beneath each board is its image under $f$.
}
\efi

See the top left board of Figure~\ref{placement 1} for an example of such an object $C$ where $m=2$. The singleton board $B=(0,0,2,3)$ fits inside $\Delta_{4,2}$. Here $k=3<4$ and there is $i=1$ white rook on the board $\Delta_{4,2}\setminus B$ and $k-i=2$ black rooks on the board $\Delta_{3,2}$ which is represented by the grey shaded cells inside $\Delta_{4,2}$. The rooks on $\Delta_{3,2}$ form a 2-level rook placement, but there is both a black rook and a white rook in the second level of $\Delta_{4,2}$.

Note that each column of $\Delta_{N,m}$ may contain at most one white rook or black rook.  On the other hand, a level of $\Delta_{N,m}$ will contain at most  one black rook, but may contain any number of white rooks.  Further, define the sign of such a placement to be $(-1)^i$. The sign of the placement on the top left in Figure~\ref{placement 1} is $-1$.

To define $I$ on an element $C\in S$, if all rooks of $C$ are in $B$, and therefore black, then $C$ is a fixed point.  Otherwise, examine the columns of $C$ from left to right until coming to a column with a rook outside of $B$. If that column contains a black rook, change the rook to a white rook, increase $i$ by one, and move every black rook above and to the right of the cell containing the new white rook down $m$ cells. If that column contains a white rook, change it to a black rook, decrease $i$ by one, and move every black rook to the right and at the same level or higher as the new rook up $m$ cells. The placement on the top right in Figure~\ref{placement 1} illustrates what happens to the board on the left under $I$. Similarly, $I$ takes the placement on the right to the placement on the left.  

We must show that $I(C)$ will be an element of $S$. Clearly each column has at most one rook. We claim that each level will still contain at most one black rook.  First, suppose that a black rook is added. In this case all black rooks at its level or above to the right of the new rook move up one level. Furthermore, there can be no black rooks at the same level or higher to the left of the new black rook. This is because the new black rook was a white rook which, by definition, was above board $B$.  Since $B$ is a singleton board, no columns of $B$ to the left of the white rook in question will terminate in the level of the white rook. Thus if there were a black rook at the same level or higher to the left, it too would be outside of board $B$, which contradicts the white rook being the leftmost rook outside of board $B$.  Thus the black rooks still form an $m$-level placement when a black rook is added.  The proof that this also holds when a black rook becomes white is similar.

We must also check that the black rooks continue to fit on the new insert board. When a black rook is added, the black rooks must be placed on a board $\Delta_{N-i+1,m}$ where the column in which the new black rook is placed is added to the columns in the initial inset $\Delta_{N-i,m}$. Since there are no white rooks to the left of the new black rook, there will be no omitted columns to the left of the column containing the new black rook, thus all cells of that column will be in the inset $\Delta_{N-i+1,m}$ and the new rook must be inside $\Delta_{N-i+1,m}$. This means that all the columns to the right of the new black rook that do not contain a white rook will contain $m$ more squares in the inset $\Delta_{N-i+1,m}$ than they did in the inset $\Delta_{N-i,m}$. Thus moving black rooks to the right of the new black rook up $m$ cells will keep them within the new $\Delta_{N-i+1,m}$. Similarly, changing a black rook to a white rook will decrease the number of cells in the columns of $\Delta_{N-i-1,m}$ to the right of the new white rook by $m$, but all black rooks to the right of the new white rook and at a higher level than it are moved down $m$ cells, so they will be in $\Delta_{N-i-1,m}$ because they were in $\Delta_{N-i,m}$ originally. Finally, if there are any black rooks below the level of the new white rook but to its right, they will remain in $\Delta_{N-i-1,m}$ because the first column in $\Delta_{N-i-1,m}$ to the right of the new white rook must go up to at least the level of the new white rook since previously it was a black rook contained in $\Delta_{N-i,m}$.  

By construction, $I$ is an involution. The fixed set of $I$ will be denoted $T$. It is the set of all configurations  which only have rooks on the subboard $B$ and, by definition, these rooks must be black. As such, $T$ is equal to the set of $m$-level rook placements of $k$ rooks on $B$. Furthermore, if a board is not in $T$, then $I$ either increases or decreases the number of white rooks on the board by one. Either way $I$ will change the sign of the board.  
And if a board is in $T$, then it has positive sign.

Given a singleton board $B'$, define $N'$, $S'$, $T'$, and $I'$ similarly for $B'$ contained in $\Delta_{N',m}$. Without a loss of generality, assume $N=N'$. Let $B'=(b'_0,b'_1,\dots,b'_{N-1})$. If $B$ and $B'$ are $m$-level rook equivalent singleton boards we can use $I$ and $I'$ to construct an explicit bijection between $m$-level rook placements of $k$ rooks on $B$ and $m$-level rook placements of $k$ rooks on $B'$. We do this by constructing a sign-preserving bijection between $S$ and $S'$.  We will need the following characterization of when two singleton boards are $m$-level rook equivalent.

The \emph{root vector} of $B$ is 
$$
\zeta_m=(-b_0,m-b_1,\dots,(N-1)m-b_{N-1}).
$$
The following result of Briggs and Remmel determines when two singleton boards are $m$-level rook equivalent simply by
considering their root vectors.

\medskip

\begin{thm}[Briggs-Remmel~\cite{br:mrn}]
\label{briggs remmel}
If $B=(b_1,\dots,b_N)$ is a singleton board then
$$
\sum_{k=0}^N r_{k,m}(B) x\da_{N-k,m} = \prod_{i=1}^N (x+b_i-(i-1)m)
$$
where $x\da_{k,m} = x(x-m)(x-2m)\dots(x-(k-1)m)$.
\end{thm}

Note that the indexing in the theorem begins at $1$, rather than $0$, simply to be consistent with the original statement of the theorem.

Since the root vector contains exactly the roots of the rook polynomial, we see that two singleton boards are $m$-level rook equivalent if and only if they have the same root vector, up to rearrangement, for a sufficiently large $N$. We are now ready to apply the Garsia-Milne Involution Principle.  

\begin{thm}
\label{rook numbers}
Let $B$ and $B'$ be $m$-level rook equivalent singleton boards. Then there exists an explicit Garsia-Milne bijection between $m$-level rook placements of $k$ rooks on $B$ and $m$-level rook placements of $k$ rooks on $B'$.
\end{thm}

\begin{proof}
By Theorem~\ref{garsia milne} and what we have already established, it suffices to find a sign-preserving bijection $f:S\rightarrow S'$. We construct $f$ as follows.

For clarity of notation, let $B$ be placed in $\Delta_{N,m}$ and $B'$ be placed in a copy $\Delta'_{N,m}$ of $\Delta_{N,m}$.  Notice that the $k$th element of the root vector of $B$, $km-b_k$,  is the number of cells in the $k$th column of $\Delta_{N,m}$ which lie outside of board $B$. Since $B$ and $B'$ are $m$-level rook equivalent, the root vector for $B'$ is a rearrangement of the root vector for $B$. Therefore there is a length-preserving bijection between the columns of the set difference $\Delta_{N,m}\setminus B$  and the columns of $\Delta'_{N,m} \setminus B'$ which takes the leftmost column of a given length in $\Delta_{N,m}\setminus B$ to the leftmost column with that length in $\Delta'_{N,m}\setminus B'$ and so forth.
This bijection induces a bijection on the placement of the white rooks. If a white rook appears in the $j$th cell above $B$, place a white rook in the $j$th cell above $B'$ in the associated column.

Once all the white rooks are placed, create a copy of $\Delta'_{N-i,m}$ inside of $\Delta'_{N,m}$ using the columns which do not contain a white rook. Place the black rooks on the board in relation to the $\Delta'_{N-i,m}$ subboard exactly as they are placed on the original board in relation to the original $\Delta_{N-i,m}$ subboard.
Each placement on the  bottom of Figure~\ref{placement 1} is the image under $f$ of the corresponding placement on the top where $B=(0,0,2,3)$ and $B'=(0,0,1,4)$. Notice that in the top left board, the white rook is at the top of the second column from the left which has two cells above $B$. In the board on the bottom left the white rook is still at the top of the second column from the left which has two cells above $B'$.

Under this map the white rooks must be placed inside $\Delta'_{N,m}$ but outside $B'$, and the black rooks are placed inside $\Delta'_{N-i,m}$, so $f$ maps $S$ to $S'$. Further this map preserves the number of white rooks placed on the board, so it is sign preserving. Therefore we may conclude from the Involution Principle that there is an explicit bijection between $m$-level rook placements of $k$ rooks on $B$ and $m$-level rook placements of $k$ rooks on $B'$.
\end{proof}

\begin{figure} \begin{center}
\begin{tikzpicture}[scale=.5]
\foreach \x in {0,1,2,3,4,5}
	\draw (\x,\x)--(6,\x);
\foreach \y in {1,2,3,4,5}
	\draw (\y,0)--(\y,\y);
\draw (6,0)--(6,5);
\draw (3.5,.5) node{$R$};
\draw (2.5,1.5) node{$R$};
\draw (4.5,2.5) node{$R$};
\draw (.5,-.5) node{$1$};
\draw (1.5,-.5) node{$2$};
\draw (2.5,-.5) node{$3$};
\draw (3.5,-.5) node{$4$};
\draw (4.5,-.5) node{$5$};
\draw (5.5,-.5) node{$6$};
\draw (6.5,.5) node{$1$};
\draw (6.5,1.5) node{$2$};
\draw (6.5,2.5) node{$3$};
\draw (6.5,3.5) node{$4$};
\draw (6.5,4.5) node{$5$};
\end{tikzpicture}
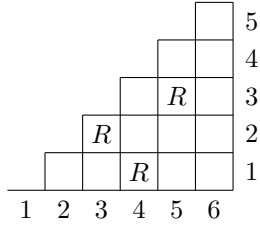
\capt{\label{Stirling} The placement corresponding to partition $\{1,4\},\{2,3,5\},\{6\}$.}
\efi

To obtain a consequence of this construction, we will need some background on symmetric functions and Stirling numbers. For $d \leq n$ both non-negative integers, let $e_d(x_1,x_2,\dots,x_n)$ denote the \emph{elementary symmetric function of degree $d$ in $n$ variables}, that is, 
\begin{equation} e_d(x_1,x_2,\dots,x_n) = \sum_{1\leq i_1<i_2<\dots<i_d\leq n}x_{i_1}x_{i_2}\dots x_{i_d}. \end{equation}
Let $S(n,d)$ denote a \emph{Stirling number of the second kind}. Recall that $S(n,d)$ can be defined as the number of ways to partition a set of $n$ elements into $d$ subsets called\emph{ blocks}.

Further, note that $S(n,d)$ counts the number of rook placements of $n-d$ rooks on $\Delta_{n,1}$. To see this, number the rows of $\Delta_{n,1}$ from 1 to $n-1$ from bottom to top. Then number the columns, including the column of height zero, from 1 to $n$ left to right. Given a partition of $\{1,\dots,n\}$ into $d$ blocks, order the elements of each block increasingly.  Now, if $i$ and $j$ are adjacent within a block then place a rook in row $i$ column $j$. See Figure~\ref{Stirling} for the rook placement corresponding to $\{1,4\},\{2,3,5\},\{6\}$. Thus the number of $m$-level rook placements of $n-d$ rooks on $\Delta_{n,m}$ is $m^{n-d}S(n,d)$. The extra $m^{n-d}$ counts the number of ways of choosing a placement for each of the $n-d$ rooks in the $m$ cells of a level.

It is interesting to note that the  construction  of $I$ yields the following theorem giving an explicit calculation for the $m$-level rook numbers of a singleton Ferrers board $B$.

\begin{thm}
For any singleton board  $B=(b_0,b_1,\dots,b_{N-1})$ 
fitting inside $\Delta_{N,m}$,
$$
r_{k,m}(B)=\sum_{i=0}^k(-1)^im^{k-i}S(N-i,N-k)e_i(-b_0,m-b_1,\dots,(N-1)m-b_{N-1}).
$$
\end{thm}

\begin{proof}
Since the fixed points of the involution $I$ are counted by $r_{k,m}(B)$, it suffices to show that the sum counts all elements of the set $S$ by sign. First note that the number of ways of putting $i$ white rooks in $i$ different columns of $\Delta_{N,m}$ outside of $B$ is $e_i(-b_0,m-b_1,\dots,(N-1)m-b_{N-1})$. Furthermore the number of $m$-level rook placements of $k-i$ rooks on $\Delta_{N-i,m}$ is $m^{k-i}S(N-i,N-k)$. Putting these two counts together with the appropriate sign gives the sum as desired.
\end{proof}

Note that this theorem implies the previously noted result that if two boards have the same root vector then they are $m$-level rook equivalent.

\subsection{Preservation of $\inv_m$}
\label{subsec:GM-maps}

\bfi
\begin{tikzpicture}[scale=.5]
\fill[lightgray] (0,0) rectangle (1,2);
\fill[lightgray] (2,0) rectangle (3,4);
\fill[lightgray] (3,0) rectangle (4,6);
\foreach \y in {1,2,3,4,5,6}
	\draw (\y-1,2*\y)--(6,2*\y);
\foreach \y in {1,2,3,4,5,6}
	\draw (\y-1,2*\y-1)--(6,2*\y-1);
\draw (0,0)--(6,0);
\foreach \x in {0,1,2,3,4,5}
	\draw (\x,0)--(\x,2*\x+2);
\draw (6,0)--(6,12);
\draw[very thick] (3,0)--(3,3);
\draw[very thick] (3,3)--(4,3);
\draw[very thick] (4,3)--(4,5);
\draw[very thick] (4,5)--(5,5);
\draw[very thick] (5,5)--(5,6);
\draw[very thick] (5,6)--(6,6);
\draw[very thick] (6,6)--(6,0);
\draw[very thick] (6,0)--(0,0);
\draw (1.5,1.5) node{$W_0$};
\draw (4.5,6.5) node{$W$};
\draw (5.5,7.5) node{$W$};
\draw (2.5,2.5) node{$R$};
\draw (3.5,.5) node{$R$};
\draw (.5,.5) node{$*$};
\draw (.5,1.5) node{$*$};
\draw (1.5,2.5) node{$*$};
\draw (1.5,3.5) node{$*$};
\draw (2.5,3.5) node{$*$};
\draw (3.5,1.5) node{$*$};
\draw (3.5,4.5) node{$*$};
\draw (3.5,5.5) node{$*$};
\draw (4.5,7.5) node{$*$};
\draw (4.5,8.5) node{$*$};
\draw (4.5,9.5) node{$*$};
\draw (5.5,8.5) node{$*$};
\draw (5.5,9.5) node{$*$};
\draw (5.5,10.5) node{$*$};
\draw (5.5,11.5) node{$*$};
\end{tikzpicture}
\hs{50pt}
\begin{tikzpicture}[scale=.5]
\foreach \x in {0,...,3}
	\fill[lightgray] (\x,0) rectangle (\x+1,2*\x+2);
\foreach \y in {1,2,3,4,5,6}
	\draw (\y-1,2*\y)--(6,2*\y);
\foreach \y in {1,2,...,6}
	\draw (\y-1,2*\y-1)--(6,2*\y-1);
\draw (0,0)--(6,0);
\foreach \x in {0,...,5}
	\draw (\x,0)--(\x,2*\x+2);
\draw (6,0)--(6,12);
\draw[very thick] (3,0)--(3,3);
\draw[very thick] (3,3)--(4,3);
\draw[very thick] (4,3)--(4,5);
\draw[very thick] (4,5)--(5,5);
\draw[very thick] (5,5)--(5,6);
\draw[very thick] (5,6)--(6,6);
\draw[very thick] (6,6)--(6,0);
\draw[very thick] (6,0)--(0,0);
\draw (1.5,1.5) node{$R$};
\draw (4.5,6.5) node{$W$};
\draw (5.5,7.5) node{$W$};
\draw (2.5,4.5) node{$R$};
\draw (3.5,2.5) node{$R$};
\draw (.5,1.5) node{$*$};
\draw (.5,.5) node{$*$};
\draw (1.5,2.5) node{$*$};
\draw (1.5,3.5) node{$*$};
\draw (2.5,5.5) node{$*$};
\draw (3.5,3.5) node{$*$};
\draw (3.5,6.5) node{$*$};
\draw (3.5,7.5) node{$*$};
\draw (4.5,7.5) node{$*$};
\draw (4.5,8.5) node{$*$};
\draw (4.5,9.5) node{$*$};
\draw (5.5,8.5) node{$*$};
\draw (5.5,9.5) node{$*$};
\draw (5.5,10.5) node{$*$};
\draw (5.5,11.5) node{$*$};
\end{tikzpicture}
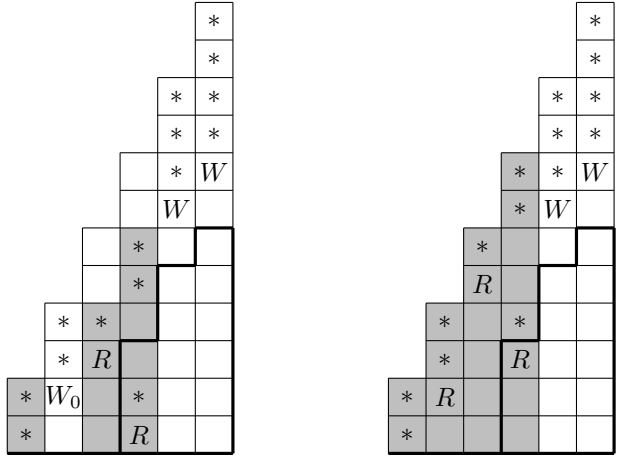
	\capt{\label{q weight GM} On the left, an element of $S$ with augmented $2$-inversion number $15$ and sign $-1$. On the right is the image of the left placement under $I$, which still has augmented $2$-inversion number $15$, but has sign $+1$.
}
\efi

We finish this section by showing that the bijections of the previous  subsection
preserve the $m$-inversion numbers of $m$-level rook placements. Given a singleton
board $B$, consider a configuration $C$ in the set $S$
consisting of $i$ white rooks in different columns
of $\Delta_{N,m}\setminus B$, together with an $m$-level placement $\pi$
of $k-i$ black rooks on the inset board $\Delta_{N-i,m}$.
Let the augmented $m$-inversion number of $C$,
denoted $\ainv_m(C)$, 
be the $m$-inversion number of $\pi$, as in Section~\ref{sec:FS}, calculated relative
to the inset board $\Delta_{N-i,m}$, plus the number of cells which lie in 
a column above a white rook. For example, the $2$-level configurations in Figure~\ref{q weight GM} both have augmented $m$-inversion number $15$.

\medskip

\begin{lem}
If $C \in S$, then $\ainv_m(C)=\ainv_m(I(C))$ where $I$ is the map from Subsection~\ref{subsec:GMrook}.
\end{lem}
\begin{proof}
It suffices to consider the case where $I$ changes the 
leftmost rook outside $B$ from white to black. In this case, let $W_0$ denote the leftmost white rook in $C$.
 All squares in the column 
above $W_0$ contributed to $\ainv_m(C)$ since they were above a white
rook. These squares must still contribute to the augmented $m$-inversion number of $I(C)$ 
because, as  proved earlier, there can be no black rook northwest of these squares. Next consider a column 
containing a black rook $R$ in $C$.
 If $R$ is to the left of $W_0$,
the cells in this column contributing to the augmented $m$-inversion number are the
same in $C$ and $I(C)$. If $R$ is to the right of $W_0$ in a lower level,
the new inset board $\Delta_{N-i+1,m}$ will have $m$ more cells above $R$
in its column. But, $m$ of those cells are located in the same level
as the new black rook where $W_0$ was, so that the
contribution of this column to $\ainv_m(C)$ is the same as to
$\ainv_m(I(C))$. If $R$ is to the right of $W_0$ at the same or
higher level, $R$ will move up $m$ cells, but the new inset board
$\Delta_{N-i+1,m}$ will also have $m$ new cells in this column.
 A similar analysis shows that
a column of the inset board containing no rook makes the same contribution
to the augmented $m$-inversion number in $C$ and $I(C)$. Finally, any column which does not intersect the inset board $\Delta_{N-i+1,m}$ contributes the same amount to $\ainv_m(C)$ and $\ainv_m(I(C))$ because none of the other white rooks have changed location.
\end{proof}

\bfi
\begin{tikzpicture}[scale=.5]
\foreach \y in {0,1,2,3,4,5,6}
    \draw (0,\y)--(3,\y);
\foreach \y in {0,2,4,6}
	\draw[very thick] (0,\y)--(3,\y);
\foreach \x in {0,1,2,3} 
   \draw (\x,0)--(\x,6);
\draw (2.5,2.5) node{$R$};
\draw (1.5, 1.5) node{$R$};
\draw (.5,4.5) node{$R$};
\end{tikzpicture}
\capt{\label{wreath} The placement on $\sqr_{3,2}$ corresponding to $(\alpha^1,\alpha^2,\alpha^1;(1,3,2))$.}
\efi

Next we show that the sign-preserving bijection $f:S\rightarrow S'$
from the proof of Theorem~\ref{rook numbers} preserves the augmented $m$-inversion number. 

\begin{lem}
If $C \in S$ and $f(C) \in S'$, then $\ainv_m(C)=\ainv_m(f(C))$ where $f$ is the map from Subsection~\ref{subsec:GMrook}. 
\end{lem}

\begin{proof}
Since $f$ sends
the placement of black rooks on the inset board $\Delta_{N-i,m}$ to
the identical placement of black rooks on the inset board $\Delta'_{N-i,m}$,
the contribution to the augmented $m$-inversion number from the black rooks is the same
in $C$ and $f(C)$. By the way $f$ moves the white rooks, the total
number of cells above the white rooks in $C$ and $f(C)$ also agrees.
Thus $\ainv_m(C)=\ainv_m(f(C))$, as desired.
\end{proof}

\begin{thm}
The explicit bijection produced by 	Theorem~\ref{rook numbers} preserves the $m$-inversion number of the board.
\end{thm}

\begin{proof}
From the previous two lemmas, we see that the Garsia-Milne Involution Principle provides a bijection $g$ between
the fixed point sets $T$ and $T'$, which preserves the augmented $m$-inversion number. Recall that a configuration $C\in T$ or $C'=g(C) \in T'$ has
no white rooks, and all black rooks are on the board $B$, not merely
on the larger board $\Delta_{N,m}$ which is the inset board when there are no white rooks. Let $\pi$ and $\pi'$ be the placments on $B$ and $B'$, respectively. Note that $\ainv_m(C)$ is computed
relative to the board $\Delta_{N,m}$, whereas $\inv_m(\pi)$
 is computed
relative to the smaller board $B$. 
 But since $B$ is a singleton Ferrers 
board located in the southeast corner of $\Delta_{N,m}$, 
every square in $\Delta_{N,m}\setminus B$ will contribute to $\ainv_m(C)$. 
Since $\#B=\#B'$, we conclude that
\[ \inv_m(\pi)=\ainv_m(C)-\#(\Delta_{N,m}\setminus B)
         =\ainv_m(C')-\#(\Delta_{N,m}\setminus B')
         =\inv_m(\pi') \]
for all $C\in T$. This shows that the bijection $g$ preserves the $m$-inversion number
of $m$-level rook placements computed relative to the boards $B$ and $B'$.
\end{proof}

\section{A bijection for hit numbers}
\label{sec:LR hit number}

We will now use the Involution Principle to prove that two boards that are $m$-level rook equivalent have the same hit numbers. We begin with some definitions. As usual, let $m$ be a fixed positive integer.

Let $B$ be a Ferrers board and let the integer $N$ be sufficiently large so that $B$ fits inside a rectangular board $\sqr_{N,m}$ with $N$ columns and $mN$ rows. If $\alpha$ is a generator of the cyclic group $C_m$,  then  
$$
C_m \wr S_N = \{(\alpha^{s_1},\alpha^{s_2},\dots,\alpha^{s_N};\sigma) \mid 
\text{$1 \leq s_i \leq m$ for each $i$ and $\sigma \in S_N$}\}.
$$ 
We associate with $\omega\in C_m \wr S_N$ a placement on $\sqr_{N,m}$ by placing a rook in level $N+1-p$ and column $i$ if $\sigma(i)=p$.  Furthermore, the rook in column $i$ will be $j$ cells from the bottom of the level if $s_i = j $. See Figure~\ref{wreath} for an example with $m=2$ and $N=3$, where the placement corresponds to $(\alpha^1,\alpha^2,\alpha^1;(1,3,2))$, and $\sigma$ is in one line notation.
 Let $R(\omega)$ denote the rook placement corresponding to $\omega$. Define the \emph{$k$th hit set} of $B$ to be
\begin{equation}
H^{(m)}_{k,N}(B)=\{R(\omega) \ | \ \omega \in C_m \wr S_N \text{ and } \#(R(\omega) \cap B) = k\}.
\end{equation}
Also define the \emph{$k$th hit number} of $B$ to be
\begin{equation}
h^{(m)}_{k,N} = \# H^{(m)}_{k,N}.
\end{equation}

In order to show that two $m$-level rook equivalent Ferrers boards have the same hit numbers, we use Garsia and Milne's result again. To do so, we must construct a signed set and a sign-reversing involution which has a set counted by $h^{(m)}_{k,N}$ as its fixed set. We do this as follows.

Let $N$ be large enough that $B$ fits inside $\sqr_{N,m}$ and fix a non-negative integer $k$ with $ k\leq N$. Then the set $S$ will consist of all configurations $C$ constructed as follows. Let $i$ vary over all non-negative integers such that $k+i \leq N$. Place $k+i$ non-attacking black, $m$-level rooks $R$ on the board $B$ if possible. If this is not possible then there are no elements of $S$ corresponding to this choice of $k$ and $i$. Furthermore, circle $i$ of the rooks in the placement. Finally, consider the $N-k-i$ columns and $N-k-i$ levels which do not contain a black rook as a subboard 
of shape $\sqr_{N-k-i,m}$.  As in the previous section, we will call this the \emph{inset $\sqr_{N-k-i,m}$ board}. Place $N-k-i$ non-attacking white $m$-level rooks, denoted by $W$, on the inset $\sqr_{N-k-i,m}$. Notice that, ignoring the color of the rooks, this is an $m$-level rook placement of $N$ rooks on $\sqr_{N,m}$. Thus it corresponds to some element of $C_m \wr S_N$. Let the sign of a configuration be $(-1)^i$.
See  Figure~\ref{hit placement} for two examples of such  configurations. Here $m=2$ and $B=(1,2,4)$ is placed fitting in $\sqr_{3,2}$.  On the left, there are no circled black rooks so $i=0$ and the white rooks are placed on the shaded inset $\sqr_{2,2}$.  On the right there is one circled black rook so $i=1$ and the white rooks are on a shaded inset $\sqr_{1,2}$.

In order to produce a sign-reversing involution $I$ on such configurations $C$, we do the following. If $B$ contains neither a white rook  nor a circled black rook, then $C$ is fixed by $I$.  Otherwise, examine the columns of $B$ from left to right until the first white rook  or circled black rook  is found. If the first rook found is white, exchange it for a circled black rook and increase $i$ by $1$. If the first rook found is a circled black rook, exchange it for a white rook and decrease $i$ by $1$.  
See  Figure~\ref{hit placement} for two examples of such configurations with $k=1$.
It is easy to see that $I$ is an involution and reverses signs in its $2$-cycles.  Also note that fixed points have no circled black rooks, so $i=0$ and the sign of the configuration is $+1$. Furthermore, for a fixed point there are no white rooks placed on $B$, so the $m$-level placement intersects $B$ in exactly $k$ black rooks. Thus the fixed points are exactly the elements of $H^{(m)}_{k,N}(B)$ if one just ignores the colors of the rooks.

\begin{figure}\begin{center}
\begin{tikzpicture}[scale=.5]
\fill[lightgray] (0,0) rectangle (2,2);
\fill[lightgray] (0,4) rectangle (2,6);
\draw (0,0)--(3,0);
\foreach \y in {0,1,2,3,4,5,6}
    \draw (0,\y)--(3,\y);
\foreach \x in {0,1,2,3} 
   \draw (\x,0)--(\x,6);
\draw[very thick] (0,0)--(3,0);
\draw[very thick] (1,2)--(2,2);
\draw[very thick] (2,4)--(3,4);
\draw[very thick] (0,0)--(0,1);
\draw[very thick] (1,1)--(1,2);
\draw[very thick] (0,1)--(1,1);
\draw[very thick] (2,2)--(2,4);
\draw[very thick] (3,0)--(3,4);
\draw (2.5,2.5) node{$R$};
\draw (1.5, 1.5) node{$W$};
\draw (.5,4.5) node{$W$};
\draw (-1,3) node{$s=$};
\end{tikzpicture}
\hs{50pt}
\begin{tikzpicture}[scale=.5]
\fill[lightgray] (0,4) rectangle (1,6);
\draw (0,0)--(3,0);
\foreach \y in {0,1,2,3,4,5,6}
    \draw (0,\y)--(3,\y);
\foreach \x in {0,1,2,3} 
   \draw (\x,0)--(\x,6);
\draw[very thick] (0,0)--(3,0);
\draw[very thick] (1,2)--(2,2);
\draw[very thick] (2,4)--(3,4);
\draw[very thick] (0,0)--(0,1);
\draw[very thick] (1,1)--(1,2);
\draw[very thick] (0,1)--(1,1);
\draw[very thick] (2,2)--(2,4);
\draw[very thick] (3,0)--(3,4);
\draw (2.5,2.5) node{$R$};
\draw (1.5,1.5) node{$R$};
\draw (.5,4.5) node{$W$};
\draw (-1,3) node{$I(s)=$};
\draw (1.5,1.5) circle (10pt);
\end{tikzpicture}
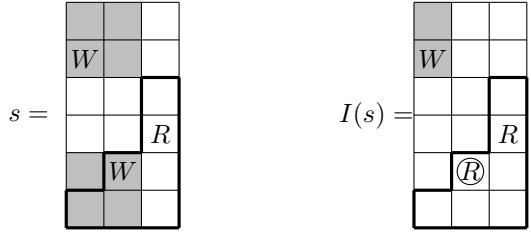
\capt{\label{hit placement} On the left, an element in $S$ with sign $+1$. On the right, the image under $I$ which has sign $-1$.}
\efi

The reader will find an example illustrating the next proof in Figure~\ref{hit place bijection}.  This example uses the boards from the example of Theorem~\ref{bijections} found in Figure~\ref{full bijection}.

\begin{thm}
\label{hit numbers}
Let $B$ and $B'$ be two $m$-level rook equivalent Ferrers boards and $N$ be large enough that $B$ and $B'$ both fit inside $\sqr_{N,m}$. Then for any non-negative integer $k \leq N$, there is an explicit bijection between $H^{(m)}_{k,N}(B)$ and $H^{(m)}_{k,N}(B')$.
\end{thm}

\begin{proof}
As in the proof of Theorem \ref{rook numbers}, we use the Garsia-Milne Involution Principle. 
Construct $S$ for $B$ placed inside $\sqr_{N,m}$ and $S'$ for $B'$ placed inside $\sqr'_{N,m}$. From what we have already done, all that remains is to construct the sign-preserving bijection $f:S\rightarrow S'$.

\begin{figure} \begin{center}
\begin{tikzpicture}[scale=.5]
\fill[lightgray] (1,6) rectangle (3,10);
\foreach \y in {0,1,2,3,4,5,6,7,8,9,10}
	\draw (0,\y)--(5,\y);
\foreach \x in {0,1,2,3,4,5}
	\draw (\x,0)--(\x,10);
\draw[very thick] (0,0)--(5,0);
\draw[very thick] (5,0)--(5,7);
\draw[very thick] (4,7)--(5,7);
\draw[very thick] (4,6)--(4,7);
\draw[very thick] (3,6)--(4,6);
\draw[very thick] (3,1)--(3,6);
\draw[very thick] (0,1)--(3,1);
\draw[very thick] (0,0)--(0,1);
\draw (.5,.5) node{$R$};
\draw (4.5,2.5) node{$R$};
\draw (3.5,5.5) node{$R$};
\draw (1.5,6.5) node{$W$};
\draw (2.5, 8.5) node{$W$};
\draw (3.5,5.5) circle (10pt);
\end{tikzpicture}
\hs{50pt}
\begin{tikzpicture}[scale=.5]
\fill[lightgray] (0,8) rectangle (2,10);
\fill[lightgray] (0,4) rectangle (2,6);
\foreach \y in {0,1,2,3,4,5,6,7,8,9,10}
	\draw (0,\y)--(5,\y);
\foreach \x in {0,1,2,3,4,5}
	\draw (\x,0)--(\x,10);
\draw[very thick] (2,0)--(2,3);
\draw[very thick] (2,3)--(3,3);
\draw[very thick] (3,3)--(3,5);
\draw[very thick] (3,5)--(4,5);
\draw[very thick] (4,5)--(4,8);
\draw[very thick] (4,8)--(5,8);
\draw[very thick] (2,0)--(5,0);
\draw[very thick] (5,0)--(5,8);
\draw (2.5,2.5) node{$R$};
\draw (3.5,.5) node{$R$};
\draw (4.5,7.5) node{$R$};
\draw (.5,4.5) node{$W$};
\draw (1.5, 8.5) node{$W$};
\draw (3.5,.5) circle (10pt);
\end{tikzpicture}
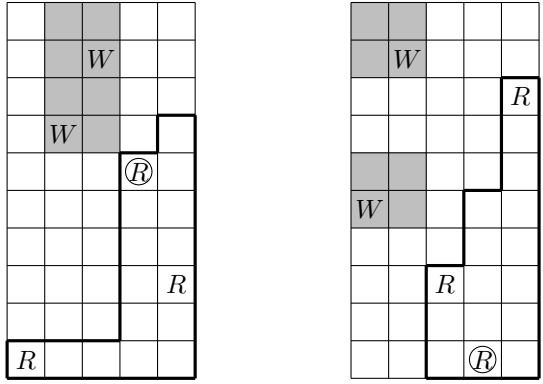
\capt{\label{hit place bijection} On the left, a $2$-level placement of white rooks, black rooks, and circled black rooks on $(1,1,1,6,7)$ inside $\sqr_{5,2}$. On the right, the corresponding placement on $(3,5,8)$ under the construction in Theorem~\ref{hit numbers}.
}
\efi

Consider an element $C \in S$. The black rooks, circled and uncircled, form an $m$-level rook placement of $k+i$ rooks on $B$. Map this to an $m$-level rook placement of $k+i$ rooks on $B'$ using the explicit bijection guaranteed by Theorem~\ref{bijections}. Furthermore, add circles to the rooks on $B'$ in such a way so that if the $r$th rook from the right on board $B$ is circled, the $r$th rook from the right on board $B'$ is circled. Finally, place the white rooks on $\sqr'_{N,m}$ by considering the inset $\sqr'_{N-k-i,m}$ of columns and levels containing no black rooks. Place the white rooks on this inset board in the exact same arrangement as they are in on the inset $\sqr_{N-k-i,m}$ of $\sqr_{N,m}$. This is easily seen to be a bijection and so the proof is complete.
\end{proof}

The next corollary follows immediately from the previous theorem.

\begin{cor}
Let $B$ and $B'$ be two $m$-level rook equivalent Ferrers boards and $N$ be large enough such that $B$ and $B'$ both fit inside $\sqr_{N,m}$. Then for any non-negative integer $k \leq N$, $h^{(m)}_{k,N}(B) = h^{(m)}_{k,N}(B')$.
\end{cor}

\section{Other Results and Open Problems}
\label{sec:Open Problem}

\subsection{A Factorization Theorem}
The $l$-operator leads to a second formulation of the factorization theorem for the $m$-level rook polynomial of a Ferrers board, originally found in~\cite{blrs:mrp}.  This theorem generalized Theorem~\ref{briggs remmel} from singleton boards to all Ferrers boards.

\begin{thm}
\label{l fact}
Let $B=(b_1,\dots,b_n)$ be a Ferrers board with $t$ non-empty levels, and let $l(B)=(l_t, l_{t-1}, \dots, l_1)$ be the singleton board where $l_p$ is the number of cells in level $p$ of $B$, as in~\S\ref{subsec:l op}. Then for any $N$ greater than or equal to both $n$ and $t$

$$
\sum_{k=0}^N r_{k,m}(B)x\da_{N-k,m} = \prod_{i=1}^N (x+l_{N-i+1}-(i-1)m)\text{,}
$$
where $l_{N-i+1}=0$ if $N-i+1 > t$.
\end{thm}

\begin{proof}
Since $r_{k,m}(B)=r_{k,m}(l(B))$ by Lemma~\ref{l operator}, the choice of $N$ ensures that
$$
\sum_{k=0}^N r_{k,m}(B)x\da_{N-k,m}=\sum_{k=0}^N r_{k,m}(l(B))x\da_{N-k,m}\text{.}
$$
Since the right hand side of the equation is the $m$-level rook polynomial of the singleton board $l(B)$, Theorem~\ref{briggs remmel}  implies
$$
\sum_{k=0}^N r_{k,m}(l(B))x\da_{N-k,m} = \prod_{i=1}^N (x+l_{N-i+1}-(i-1)m)\text{.}
$$
Combining these equations yields the desired theorem.
\end{proof}

\subsection{Open Problems}
The characterization of the rook equivalence class of a singleton board in terms of its root vector in Theorem~\ref{briggs remmel} provides a way to count the number of singleton boards in a given $m$-level rook equivalence class as was done in~\cite{blrs:mrp}. However, it is an open problem to count the total number of Ferrers boards in a given $m$-level rook equivalence class.  If $C$ is a singleton board, then perhaps counting the number of Ferrers boards $B$ with $l(B)=C$ would be a good start to this problem, but this too remains open.

Another open question concerns a $p$-analogue of the $m$-level rook numbers. 
Here, $p$ refers to a variable and not a level.
When the $q$-analogue was introduced above, it was mentioned that Briggs and Remmel assigned to each $m$-level rook placement $\pi$ a monomial in $p$ and $q$, where the power of $q$ turned out to be $\inv_m(\pi)$. The interpretation of the power of $p$ turns out to be less intuitive. Given a placement $\pi$ of $k$ rooks in columns $c_1,c_2,\dots,c_k$, let $\beta(\pi)$ denote the number of cells $c$ satisfying the following conditions:
\begin{enumerate}
\item The cell $c$ is below a cell containing a rook.
\item There is no rook to the left of $c$ in the same level.
\end{enumerate}
Then the power of $p$ associated with placement $\pi$, called the \emph{$p$-weight of $\pi$}, is 
$$
\wt_m(\pi)=\beta(\pi)-m(c_1+c_2+\dots+c_k).
$$ 
For example, the $2$-level rook placement on the right in Figure~\ref{p weight} has $\beta(\pi)=1$ and $c_1=1$
so that  $\wt_2(\pi)=1-2(1)=-1$.

\bfi
\begin{tikzpicture}[scale=.5]
\foreach \x in {0,1,2}
	\draw (\x,0)--(\x,1);
\foreach \y in {0,1}
	\draw (0,\y)--(2,\y);
\draw (.5,.5) node{$R$};
\draw (-1,.5) node{$B=$};
\end{tikzpicture}
\hs{50pt}
\begin{tikzpicture}[scale=.5]
\foreach \x in {0,1}
	\draw (\x,0)--(\x,2);
\foreach \y in {0,1,2}
	\draw (0,\y)--(1,\y);
\draw (.5,1.5) node{$R$};
\draw (-1.2,1) node{$l(B)=$};
\draw (.5,.4) node{*};
\end{tikzpicture}
\capt{\label{p weight} On the left, a $2$-level placement $\pi$ of one rook on Ferrers board $B=(1,1)$ with $\wt_2(\pi)=-2$.  On the right, $l(B)$ with $\wt_2(l(\pi))=-1$. The single cell counted by $\beta(l(\pi))$ is denoted with an asterisk.
}
\efi

Unfortunately, the bijections on rook placements given in this paper do not preserve the $p$-weight of a placement. In fact, the multiset of $p$-weights associated with two $m$-level rook equivalent Ferrers boards may not even be equal. Consider, for example, $m=2$ and the boards $B=(1,1)$ and $l(B)=(2)$, shown in Figure~\ref{p weight}. Clearly the two boards are $m$-level rook equivalent, because the board on the right is obtained by applying the $l$-operator to the board on the left. 
On either board, there is one way to place no rooks and two ways to place one rook.  Doing this on $B$ yields $p$-weights of $0,-2,-4$ but on $l(B)$ one obtains $0, -1,-2$.

This leads to a few related open questions. Is there another $p$-analogue of $m$-level rook placements which is preserved by the bijections given in this paper, or by similar bijections? If such a $p$-analogue exists, does it have a more ``natural'' motivation? Also, is there a factorization of the $p,q$-analogue of the  $m$-level rook polynomial using the new $p$-analogue, similar to that given in~\cite{blrs:mrp}?

Finally there are some open problems related to hit numbers. In~\cite{br:mrn}, Briggs and Remmel defined the $p,q$-hit numbers 
$h^{(m)}_{n,k}(B,p,q)$ for any singleton board $B$ that fits inside the 
rectangular board $\mathrm{Sq}_{n,m}$ by 
\begin{equation}\label{hitdef}
\sum_{k=0}^n h^{(m)}_{k,n}(B,p,q)x^k = \sum_{k=0}^n r_{k,m}(B,p,q) 
[m(n-k)]_{\downarrow_{n-k,m}}p^{m\left(\binom{k+1}{2}+k(m-k)\right)}
\prod_{\ell = n-k+1}^n (x - q^{m\ell}p^{m(n-\ell)}),
\end{equation}
where $r_{k,m}(B,p,q)$ is the $p,q$-rook number defined in 
\cite{br:mrn}, $[n]_{p,q} =\frac{p^n-q^n}{p-q}=p^{n-1}+p^{n-2}q + \cdots 
+ pq^{n-2}+q^{n-1}$ for any 
positive integer $n$, and $[mk]_{\downarrow_{k,m}} =[mk]_{p,q}[m(k-1)]_{p,q} 
\cdots [m]_{p,q}$.  They showed 
that for all singleton boards $B$ that fit inside the 
rectangular board $\mathrm{Sq}_{n,m}$, $h^{(m)}_{n,k}(B,p,q)$ is 
always polynomial in $p$ and $q$ with non-negative integer coefficients. 
In \cite{b:thesis}, Briggs gave a combinatorial 
interpretation of the $h^{(m)}_{k,n}(B,1,q)$ for any 
singleton Ferrers board $B$ that fits inside the 
rectangular board $\mathrm{Sq}_{n,m}$ as follows: 
\begin{equation}
h^{(m)}_{k,n}(B,1,q) = \sum_{R(\omega) \in H^{(m)}_{k,n}} 
q^{\xi^m_B(R(\omega))}
\end{equation}
where $\xi^m_B(R(\omega))$ can be calculated for any $R(\omega)$ as follows,
\begin{enumerate}
\item each rook $R$ that does 
not lie in $B$ cancels all the cells in its column that lie weakly below 
$R$ and outside of $B$ plus all the cells in its level 
which lie strictly to the right of $R$, and
\item each rook
$R$ that lies in $B$ cancels all the cells in its column that either lie weakly below 
$R$ or outside of $B$, and all the cells in its level 
which lie strictly to the right of $R$,  and
\item $\xi^m_B(R(\omega))$ is the number of uncanceled cells in $\mathrm{Sq}_{n,m}$.
\end{enumerate}
For example, Figure~\ref{xi} shows a case where $m =3$ and $n=4$. The placement is an element $R(\omega) \in H^{(3)}_{2,4}$. We have 
put asterisks in all the cells which are canceled, which do not already contain rooks, so that 
$\xi^3_B(R(\omega)) =9$. 
In the special case $m=1$, this statistic corresponds 
to the statistic for hit numbers on Ferrers boards due to Dworkin 
\cite{d:igr}.

\bfi
\begin{tikzpicture}[scale=.5]
\foreach \y in {0,1,2,3,4,5,6,7,8,9,10,11,12}
	\draw (0,\y)--(4,\y);
\foreach \x in {0,1,2,3,4}
	\draw (\x,0)--(\x,12);
\draw[very thick] (0,0)--(4,0);
\draw[very thick] (4,0)--(4,10);
\draw[very thick] (3,10)--(4,10);
\draw[very thick] (3,6)--(3,10);
\draw[very thick] (2,6)--(3,6);
\draw[very thick] (2,4)--(2,6);
\draw[very thick] (1,4)--(2,4);
\draw[very thick] (1,2)--(1,4);
\draw[very thick] (0,2)--(1,2);
\draw[very thick] (0,0)--(0,2);
\draw (.5,10.5) node{$R$};
\draw (1.5,2.5) node{$R$};
\draw (2.5, 7.5) node{$R$};
\draw (3.5,4.5) node{$R$};
\draw (.5,2.5) node{*};
\draw (.5,3.5) node{*};
\draw (.5,4.5) node{*};
\draw (.5,5.5) node{*};
\draw (.5,6.5) node{*};
\draw (.5,7.5) node{*};
\draw (.5,8.5) node{*};
\draw (.5,9.5) node{*};
\draw (1.5,.5) node{*};
\draw (1.5,1.5) node{*};
\draw (1.5,4.5) node{*};
\draw (1.5,5.5) node{*};
\draw (1.5,6.5) node{*};
\draw (1.5,7.5) node{*};
\draw (1.5,8.5) node{*};
\draw (1.5,9.5) node{*};
\draw (1.5,10.5) node{*};
\draw (1.5,11.5) node{*};
\draw (2.5,.5) node{*};
\draw (2.5,1.5) node{*};
\draw (2.5,2.5) node{*};
\draw (2.5,6.5) node{*};
\draw (2.5,9.5) node{*};
\draw (2.5,10.5) node{*};
\draw (2.5,11.5) node{*};
\draw (3.5,.5) node{*};
\draw (3.5,1.5) node{*};
\draw (3.5,2.5) node{*};
\draw (3.5,3.5) node{*};
\draw (3.5,6.5) node{*};
\draw (3.5,7.5) node{*};
\draw (3.5,8.5) node{*};
\draw (3.5,9.5) node{*};
\draw (3.5,10.5) node{*};
\draw (3.5,11.5) node{*};
\end{tikzpicture}
\capt{\label{xi} An example on $B=(2,4,6,10)$ with $m=3$ and $n=4$. The cancelled cells which do not contain rooks are marked with asterisks. Note there are $9$ empty cells which are uncancelled, so $\xi_B^3(R(\omega))=9$.
}
\efi

Our bijection $\theta$ of Theorem~\ref{hit numbers} between between $H^{(m)}_{k,N}(B)$ and 
$H^{(m)}_{k,N}(B')$ does not send the statistic $\xi^m_B$ 
to the statistic $\xi^m_{B'}$.  That is, it is not alway the case 
that if $R(\omega) \in H^{(m)}_{k,N}(B)$, then 
$\xi_B^m(R(\omega)) = \xi_{B'}^m(\theta(R(\omega)))$. 
Thus we ask whether it is poosible to 
define a natural bijection $\Gamma$ between $H^{(m)}_{k,N}(B)$ and 
$H^{(m)}_{k,N}(B')$ such that  
$\xi_B^m(R(\omega)) = \xi_{B'}^m(\Gamma(R(\omega)))$? Also, 
if we use (\ref{hitdef}) to define $h^{(m)}_{k,n}(B,p,q)$ for 
non-singleton Ferrers boards contained in $\mathrm{Sq}_{n,m}$, can we classify 
the collection of such boards such that $h^{(m)}_{k,n}(B,p,q)$ is 
always a polynomial in $p$ and $q$ with non-negative integer coefficients, or 
when $h^{(m)}_{k,n}(B,1,q)$ is 
always polynomial in $q$ with non-negative integer coefficients?

\bibliographystyle{alpha}
\bibliography{kenny}
\label{sec:biblio}

\end{document}